\documentclass[11pt]{article}
\usepackage{eurosym}
\usepackage{amssymb}
\usepackage{amsfonts}
\usepackage{amsmath}
\usepackage{graphicx}
\usepackage{hyperref}
\usepackage{subcaption}
\usepackage{cancel}
\usepackage{tikz}
\usepackage{float}
\usepackage{cite}
\usepackage{float}

\newtheorem{theorem}{Theorem}

\newtheorem{lemma}[theorem]{Lemma}

\newtheorem{proposition}[theorem]{Proposition}
\newtheorem{remark}[theorem]{Remark}

\newtheorem{assumption}[theorem]{Assumption}
\newenvironment{proof}[1][Proof]{\textbf{#1.} }{\ \rule{0.5em}{0.5em}}
\newcommand{\bh}{\mathbf{H}}
\begin{document}

\author{Pelin G. Geredeli \\
%EndAName
School of Mathematical and Statistical Sciences,\\ Clemson University, Clemson-SC 29634, USA \\ pgerede@clemson.edu }
\title{Spectral Analysis and Asymptotic Decay of the Solutions to Multilayered Structure-Stokes Fluid Interaction PDE System
}
\maketitle

\begin{abstract}

In this work, the dynamics of a multilayered structure-fluid interaction (FSI) PDE system is considered. Here, the coupling of 3D Stokes and 3D elastic dynamics is realized via an additional 2D elastic equation on the boundary interface. Such modeling PDE systems appear in the mathematical modeling of eukaryotic cells and vascular blood flow in mammalian arteries.  We analyze the long time behavior of solutions to such FSI coupled system in the sense of strong stability. 

Our proof is based on an analysis of the spectrum of the associated semigroup generator $\mathbf{A}$ which in particular entails the elimination of all three parts of the spectrum of $\mathbf{A}$ from the imaginary axis. In order to avoid steady states in our stability analysis, we firstly show that zero is an eigenvalue for the operator $\mathbf{A}$, and we provide a characterization of the (one dimensional) zero eigenspace $Null(\mathbf{A}).$ In turn, we address the issue of asymptotic decay of the solution to the zero state for any initial data taken from the orthogonal complement of the zero eigenspace $Null(\mathbf{A})^{\bot}.$

\vskip.3cm \noindent \textbf{Key terms:} Fluid-Structure Interaction, Stokes Flow, Stability  

\vskip.3cm \noindent \textbf{Mathematics Subject Classification (MSC): 35B35, 35B40, 35G35}

\end{abstract}

\section{Introduction}

\hspace{0.6cm}The multilayered (composite) structure-Stokes flow interaction PDE model which we consider in this manuscript arises in the context of the blood transportation process in mammalian arteries and shape deformation of cells in cellular dynamics. The physiological interaction between arterial walls (or the membranes of the eukaryotic cells) and blood flow (or the surrounding fluids) plays a crucial role in the physiology and pathophysiology of the human cardiovascular system and cells \cite{ap-1, FSIforBIO, hsu, BorisSimplifiedFSI, SunBorMulti}, and can be mathematically realized by fluid structure interaction (FSI) PDE. In particular, all these dynamics are governed via multi-physics PDE models.

In the existing literature of FSI dynamics, it is often the case that only a single layer structural PDE component is present in the model; i.e., the displacement along the interaction interface is not modeled via any elastic equation. For example, coupled heat-wave PDE systems (and some variations) have been considered in the past where the heat equation component is regarded as a simplification of the fluid flow component of the fluid-structure interaction (FSI) dynamics, and the wave equation is regarded as a simplification of the structural (elastic) component; see e.g., \cite[Section 9]{lions1969quelques} and \cite{RauchZhangZuazua}. Also, the FSI dynamics in which the fluid PDE component of fluid-structure interactions is governed by Stokes or Navier-Stokes flow were studied in various settings \cite{av-trig, AvalosTriggiani09, Barbu, Chambolle, Courtand, Du, KochZauzua, lions1969quelques}.  However, some degree of physical realism is lost if an arterial wall is taken to have no composite layers. Because the vascular wall structures are typically not single-layered; they are in fact of multi-layered type (\cite{multi-layered, buk, SunBorMulti}).

The first contribution to the analysis of such multilayered FSI dynamics is the pioneering paper \cite{SunBorMulti}, where multilayered FSI is composed of 2D (thick layer) wave equation and 1D wave equation (thin layer) coupled to a 2D fluid PDE across a boundary interface. In  \cite{SunBorMulti}, the authors proved the wellposedness of the said multilayered coupled system using a partitioned, loosely coupled scheme. This work represented a crucial result for future studies, since it showed that the presence of a thin structure with mass at the fluid–structure interface indeed regularizes the FSI dynamics. Later on,  the authors in \cite{AGM}, considered a multilayered ``canonical" 3D heat-2D wave-3D wave coupled system with the objective of analyzing the existence-uniqueness and the asymptotic behavior of the corresponding solutions. In \cite{AGM}, although the considered multilayered heat-wave-wave system constitutes a simplification somewhat of the FSI model in \cite{SunBorMulti}– in particular, the 2D wave equation takes the place of a fourth order plate or shell PDE – our results in \cite{AGM} still remained valid if one replaces the 2D wave equation with the corresponding linear fourth order equation. With an objective of better understanding the regularizing effects of the underlying fluid dissipation and the coupling effects at the boundary interface in higher dimension, the authors of \cite{AGM} analyzed the wellposedness and stability properties of the said multilayered FSI PDE system. In \cite{AGM}, to obtain their results, the authors appealed to the semigroup approach, and the spectral theory of linear operators.

Subsequently, in \cite{AA}, a multilayered Lam$\acute{e}$-heat PDE system was considered, and the authors gave an alternate proof for the analogous strong stability result which they obtained in \cite{AGM}. The main reasoning to pursue a different approach then that in \cite{AGM} was to eschew the need for deriving analogous energy identities for the thick Lam$\acute{e}$ solution component of the multilayered structure-fluid interaction PDE system. To this end, the strong stability proof in \cite{AA} was predicated upon ultimately invoking the resolvent criterion introduced by Y. Tomilov \cite{tomilov}.

The strong stability result obtained in \cite{AGM} for the solution to the multilayered ``canonical" 3D heat - 2D wave - 3D wave coupled system has been improved in \cite{RD}, where the authors proved that  this solution actually satisfies a rational decay rate with respect to smooth initial data.

Very recently, in \cite{Ger}, a more realistic multilayered FSI PDE model was analyzed for which the canonical heat component was replaced with the Stokes flow and coupling of the 3D Stokes flow and 3D elastic (structural) PDE components is realized, via an additional 2D elastic system on the boundary interface. The main result of \cite{Ger} was the semigroup wellposedness of the coupled Stokes-wave-Lam$\acute{e}$ PDE model. Unlike the papers mentioned above, dealing with the pressure term (due to the Stokes fluid) was one of the main challenges–not seen in the aforecited works– in determining the existence and uniqueness of solutions, and this term indeed required a nonstandard elimination of the pressure term for semigroup generation of the corresponding dynamical system. This difficulty was overcome by expressing the fluid solution variable via a decoupling of the Stokes equation, and constructing the elastic solution variables by solving a mixed variational formulation via a Babuska-Brezzi approach.

In the current manuscript, our main goal is to analyze the long time behavior of the solution to the 3D Stokes flow-2D elastic-3D elastic coupled PDE system whose existence-uniqueness was established in \cite{Ger}. In this regard, we show that the associated semigroup is strongly stable. That is, there is asymptotic decay to the zero state. \textit{To the best of the author's knowledge there is no result about the strong stability of such multilayered structure-Stokes fluid interaction PDE systems with an elastic interface.} 

At this point, we should point out the main challenges encountered throughout the manuscript and novelties:\\

\textbf{(i)} \textbf{Spectral Analysis for the generator $\mathbf{A}:$} In order to derive our stability result, we adopt a frequency domain approach. Since the domain $D(\mathbf{A})$ of the associated $C_0-$semigroup generator $\mathbf{A}$ (see \ref{gener}) is not compactly embedded into finite energy space $\mathbf{H}$ (see \ref{H} below), classical weak stabilizability approaches are not availing; see \cite{bench}. Instead, we appeal to well known strong stability criterion introduced by Arend-Batty \cite{A-B}. However, the spectral analysis required for the generator $\mathbf{A}$ entails a great amount of technicalities. For example, one needs to analyze all parts of the spectrum $\sigma(\mathbf{A}),$ and show that it does not intersect the imaginary axis. \\

\textbf{(ii)} \textbf{Zero eigenvalue for the generator $\mathbf{A}:$} Since there is a need to avoid steady states in our stability analysis, as a prerequisite result, we show that zero is an eigenvalue for the generator $\mathbf{A}$ of the 3D Stokes flow-2D elastic-3D elastic coupled PDE system. This gives us an alert that the presence of the zero eigenvalue disallows the consideration of any sort of decay problem in the entire phase space $\mathbf{H}$. Accordingly, we are led here to consider our strong stability analysis on the orthogonal complement $\mathbf{N}^{\bot}$ of the zero eigenspace $\mathbf{N}$ that is defined in \eqref{comp} below. In this regard, we provide clean characterization for the zero eigenspace $\mathbf{N},$ which is only one dimensional, and the orthogonal complement $\mathbf{N}^{\bot}-$ultimately, the Hilbert space of stabilizable initial data.\\

\textbf{(iii)} \textbf{Overdetermined eigenvalue problem for the thick 3D-elastic component:} As we mentioned in (i), our strong stability proof entails the elimination of all three parts of the spectrum of the generator $\mathbf{A}$ from the imaginary axis. Accordingly, we analyze $\sigma_p(\mathbf{A})$ (point spectrum), $\sigma_r(\mathbf{A})$ (residual spectrum), and $\sigma_c(\mathbf{A})$ (continuous spectrum). For this, we are essentially interested in static equations with respect to the operators $\mathbf{A}$ and $(i\beta I-\mathbf{A})$ for some $\beta \neq 0.$ In order to solve these resolvent equations, we apply certain multipliers to obtain necessary estimates to solve the solution variables. However, getting an estimate for the 3D-thick elastic solution variable $\{w_{0n}\}$ complicates our analysis. In particular, we see that the static equation obtained for $\{w_{0n}\}$ represents an overdetermined eigenvalue problem. In order to solve such a problem, we are forced to impose an assumption (see Assumption \ref{assump}) to guarantee that the spectrum of the associated generator $\mathbf{A}$ remains away from the imaginary axis. Such (essentially geometric) assumptions have been seen in the context of stabilization problems for other FSI. See e.g., \cite{Av, GR, av-trig}. 

\newpage

\noindent \textbf{Description of the Multilayered FSI PDE Model}
\vspace{0.5cm}

\noindent We describe the mathematical modeling of the multilayered (composite) structure-Stokes fluid interaction PDE system on the following geometry: the fluid geometry $\Omega _{f}$ $\subseteq \mathbb{R}^{3}$ is a
bounded domain of class $C^3$ with exterior boundary $\Gamma _{f}$. The
structure domain $\Omega _{s}$ $\subseteq \mathbb{R}^{3}$ is
\textquotedblleft completely immersed\textquotedblright\ in $\Omega _{f}$ (See figure below) 

\begin{center}
\includegraphics[scale=0.4]{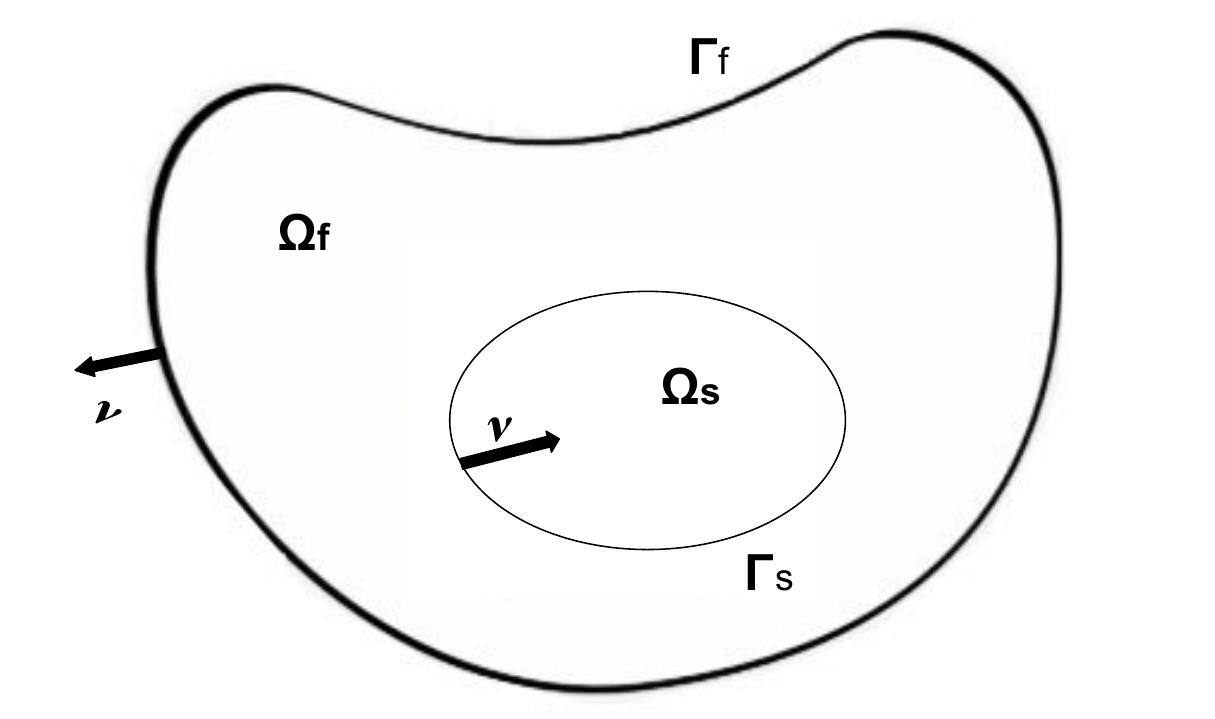}

\textbf{Figure: Geometry of the FSI Domain}
\end{center}

\noindent The boundary $\Gamma _{s}=\partial \Omega _{s},$ between $\Omega _{f}$ and $%
\Omega _{s}$ is of class $C^2,$ and constitutes the interactive interface between the fluid and
structure dynamics. In addition, $\nu (.)$ is the unit normal vector which is outward with
respect to $\Omega _{f},$ and so inwards with respect to $\Omega _{s}.$ 
\vspace{0.5cm}

\noindent With the geometry $%
\{\Omega _{s},\Omega _{f}\}$ above, the PDE system in variables $[u,h,h_t,w,w_t]$ is:

\begin{equation}
\left\{ 
\begin{array}{l}
u_{t}-\text{div}(\nabla u+\nabla ^{T}u)+\nabla p=0\text{ \ \ \ in \ }%
(0,T)\times \Omega _{f} \\ 
\text{div}(u)=0\text{ \ \ \ \ \ \ \ \ \ \ \ \ \ \ \ \ \ \ \ \ \ \ \  \ \ \ \ \ \  in \ }(0,T)\times \Omega _{f} \\ 
u|_{\Gamma _{f}}=0\text{ \ \ \ \ \ \ \ \ \ \ \ \ \ \ \ \ \ \ \ \ \ \ \  \ \ \ \ \ \  \ \ \ on \ }(0,T)\times \Gamma _{f};%
\end{array}%
\right.   \label{2a}
\end{equation}
\begin{equation}
h_{tt}-\Delta _{\Gamma _{s}}h=[\nu \cdot
\sigma (w)]|_{\Gamma _{s}}-[\nu \cdot (\nabla u+\nabla ^{T}u)]|_{\Gamma
_{s}}+p \nu \text{ \ \ \ on \ }(0,T)\times \Gamma _{s}, \label{2.5b}
\end{equation}
\begin{equation}
\left\{ 
\begin{array}{l}
w_{tt}-\text{div}\sigma (w)+w=0\text{ \ \ \ \ \ \ \ \ \ \ \ \  \ on \ }(0,T)\times \Omega _{s} \\ 
w_{t}|_{\Gamma _{s}}=h_{t}=u|_{\Gamma _{s}}\text{ \ \ \ \ \ \ \ \ \ \ \ \ \ \ \ \ \ \ \ \ on \ }(0,T)\times
\Gamma _{s}%
\end{array}%
\right.   \label{2d}
\end{equation}%
\begin{equation}
\lbrack u(0),h(0),h_{t}(0),w(0),w_{t}(0)]=[\widetilde{u}_{0},\widetilde{h}_{0},\widetilde{h}_{1},\widetilde{w}_{0},\widetilde{w}_{1}]\in 
\mathbf{N}^{\bot}.  \label{IC}
\end{equation}%
Here, $\mathbf{N}$ represents the zero eigenspace of the semigroup generator which is associated to the above PDE system (see the definition of $\mathbf{A}:\mathbf{H} \rightarrow \mathbf{H}$ in \eqref{gener} below.) Also, $\lambda=0 $ is indeed an eigenvalue of $\mathbf{A}$, see Theorem \ref{eigen} below. Moreover, orthogonal complement $\mathbf{N}^{\bot}$ is characterized as
\begin{equation}\mathbf{N}^{\bot}=\{[u_0,h_{0},h_{1},w_{0},w_{1}]\in \mathbf{H}: \int_{\Gamma_s} (\nu \cdot h_0) d\Gamma_s =0\},\label {complement}\end{equation}
(see Theorem \ref{eigen} below,) where 
\begin{equation}
\begin{array}{l}
\mathbf{H}=\mathbf{N}\bigoplus \mathbf{N}^{\bot}=\{[u_{0},h_{0},h_{1},w_{0},w_{1}]\in [L^{2}(\Omega _{f})]^3\times
[H^{1}(\Gamma _{s})]^2\times [L^{2}(\Gamma _{s})]^2\times \\ 
\text{ \ \ \ \ \ \ \ \ \ \ \ }\times [H^{1}(\Omega _{s})]^3\times [L^{2}(\Omega
_{s})]^3: \text{ div}(u_{0})=0,\text{ \ \ }u_{0}\cdot \nu |_{\Gamma
_{f}}=0\text{, ~and~~ }w_{0}|_{\Gamma _{s}}=h_{0}\text{ }\}%
\end{array}
\label{H}
\end{equation}%
is the associated finite energy Hilbert space with the inner product%
\begin{equation*}
\left\langle\Phi _{0},\widetilde{\Phi }_{0}\right\rangle_{\mathbf{H}} =\left\langle u_{0},\widetilde{u}%
_{0}\right\rangle_{\Omega _{f}}+\left\langle\nabla _{\Gamma _{s}}(h_{0}),\nabla _{\Gamma _{s}}(%
\widetilde{h}_{0})\right\rangle_{\Gamma _{s}}+\left\langle h_{1},\widetilde{h}_{1}\right\rangle_{\Gamma _{s}}
\end{equation*}
\begin{equation}
\text{\ \ \ \ \ \ \ \ \ }+\left\langle\sigma (w_{0}),\epsilon (\widetilde{w}_{0})\right\rangle_{\Omega _{s}}+\left\langle w_{0},%
\widetilde{w}_{0}\right\rangle_{\Omega _{s}}+\left\langle w_{1},\widetilde{w}_{1}\right\rangle_{_{\Omega _{s}}},
\label{Hilbert}
\end{equation}
for any
\begin{equation}
\Phi _{0}=\left[ u_{0},h_{0},h_{1},w_{0},w_{1}\right] \in \mathbf{H}\text{;
\ }\widetilde{\Phi }_{0}=\left[ \widetilde{u}_{0},\widetilde{h}_{0},%
\widetilde{h}_{1},\widetilde{w}_{0},\widetilde{w}_{1}\right] \in \mathbf{H}.
\label{stat}
\end{equation}
We also note here that $\Delta _{\Gamma _{s}}(.)$ is the Laplace Beltrami operator, and the stress tensor $\sigma(.) $ constitutes the Lam$\acute{e}$ system of elasticity on
the ``thick" layer. Namely, for function $v$ in $\Omega _{s},$ 
\[
\sigma (v)=2\mu \epsilon (v)+\lambda \lbrack I_{3}\cdot \epsilon (v)]I_{3},
\]%
where strain tensor $\epsilon (\cdot )$ is given by 
\[
\epsilon _{ij}(v)=\frac{1}{2}\left( \frac{\partial v_{j}}{\partial x_{i}}+%
\frac{\partial v_{i}}{\partial x_{j}}\right) ,\text{ \ \ }1\leq i,j\leq 3.
\]
\begin{remark}
 For the sake of numerical computation, the structure geometry $\Omega _{s}\subset 
%TCIMACRO{\U{211d} }%
%BeginExpansion
\mathbb{R}
%EndExpansion
^{3},$ can also be taken to be a convex polyhedral domain with polygonal
boundary faces $\Gamma _{j},$ $1\leq j\leq K$, where $\Gamma _{i}\cap \Gamma
_{j}\neq \emptyset $ for $i\neq j,$ and, 
\[
\Gamma _{s}=\cup _{j=1}^{K}\overline{\Gamma }_{j}.
\]%

\begin{center}
\includegraphics[scale=0.4]{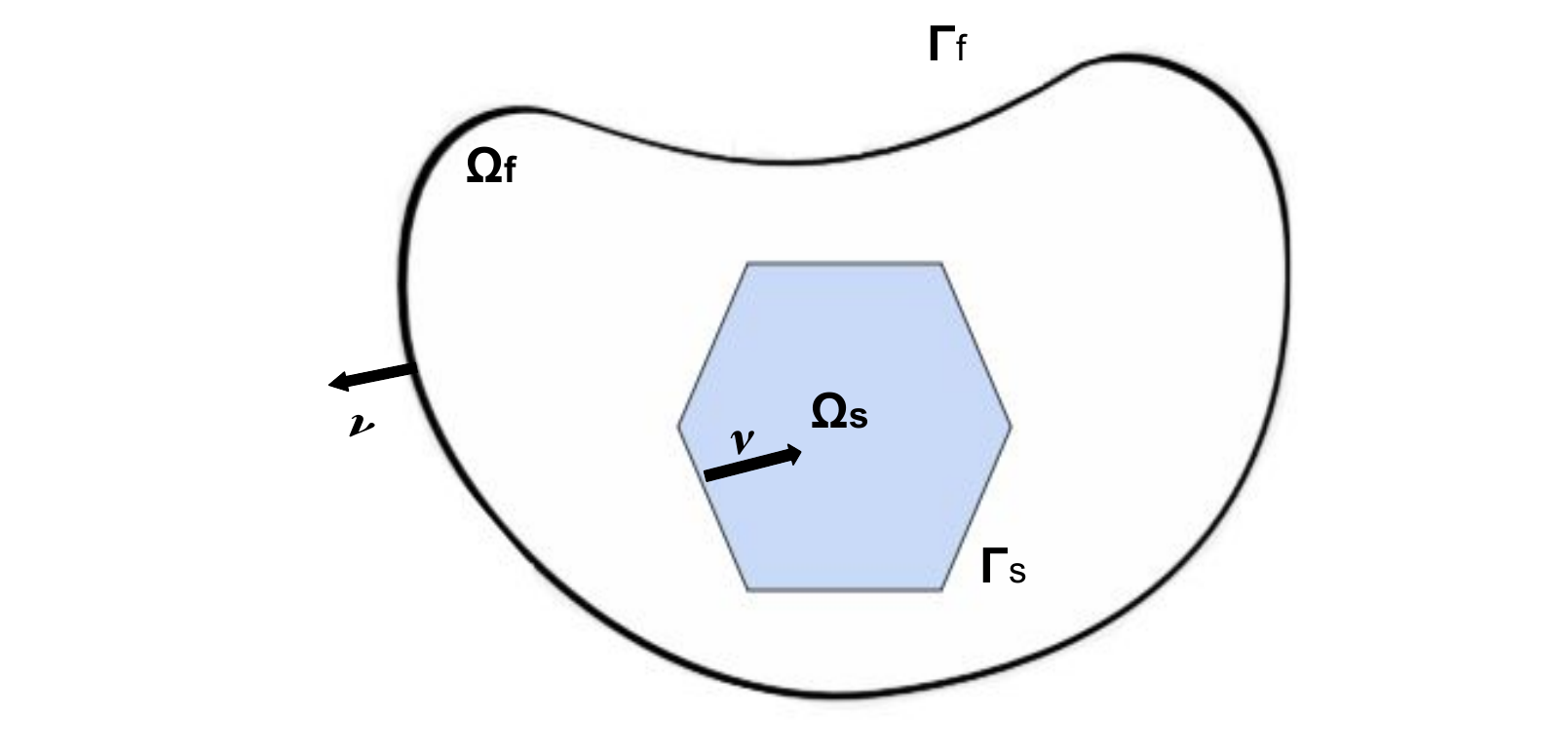}

\textbf{Figure: Alternative Geometry of the FSI Domain}
\end{center}
In this case, the thin wave equation can be modeled for $j=1,...,K$ as%
\[
\left\{ 
\begin{array}{l}
\frac{\partial ^{2}}{\partial t^{2}}h_{j}-\Delta h_{j}=[\nu \cdot
\sigma (w)]|_{\Gamma _{j}}-[\nu \cdot (\nabla u+\nabla ^{T}u)]|_{\Gamma
_{j}}+p \nu 
\text{ \ \ \ on \ }(0,T)\times \Gamma _{j} \\ 
h_{j}|_{\partial \Gamma _{j}\cap \partial \Gamma _{l}}=h_{l}|_{\partial
\Gamma _{j}\cap \partial \Gamma _{l}}\text{ on \ }(0,T)\times (\partial
\Gamma _{j}\cap \partial \Gamma _{l})\text{,\ } \forall~ 1\leq l\leq K\text{
s.t. }\partial \Gamma _{j}\cap \partial \Gamma _{l}\neq \emptyset  \\ 
\left. \dfrac{\partial h_{j}}{\partial n_{j}}\right\vert _{\partial \Gamma
_{j}\cap \partial \Gamma _{l}}=-\left. \dfrac{\partial h_{_{l}}}{\partial
n_{l}}\right\vert _{\partial \Gamma _{j}\cap \partial \Gamma _{l}}\text{on \ 
}(0,T)\times (\partial \Gamma _{j}\cap \partial \Gamma _{l})\text{,\ } \forall~ %
1\leq l\leq K\text{ s.t }\partial \Gamma _{j}\cap \partial \Gamma
_{l}\neq \emptyset .\text{\ }%
\end{array}%
\right. 
\]%
where the Laplace Beltrami Operator $\Delta _{\Gamma _{s}}(.)$ in \eqref{2.5b} is replaced with the standard Laplace
operator with the imposition of additional continuity and boundary conditions in
order to satisfy the surface differentiation (See \cite{AGM, RD}.)
   
\end{remark}
\subsection{Notation}

For the remainder of the text, norms $||\cdot ||_{D}$ are taken to be $%
L^{2}(D)$ for the domain $D$. Inner products in $L^{2}(D)$ are written $%
<\cdot ,\cdot >_{D}$, and the inner products $L_{2}(\partial D)$ are written $%
\langle \cdot ,\cdot \rangle_{\partial D} $. The space $H^{s}(D)$ will denote the Sobolev
space of order $s$, defined on a domain $D$, and $H_{0}^{s}(D)$ denotes the
closure of $C_{0}^{\infty }(D)$ in the $H^{s}(D)$ norm which we denote by $%
\Vert \cdot \Vert _{H^{s}(D)}$ or $\Vert \cdot \Vert _{s,D}$. 

\noindent For the sake of simplicity, we denote $$\mathbf{L^2}(\Omega_f)=[L^{2}(\Omega _{f})]^3,~~\mathbf{L^2}(\Omega_s)=[L^{2}(\Omega
_{s})]^3,~~\mathbf{H^1}(\Omega_s)=[H^{1}(\Omega _{s})]^3,~~\mathbf{H^1}(\Gamma_s)=[H^{1}(\Gamma _{s})]^2.$$ 
Also, we make use of the standard notation for the trace of functions defined on a Lipschitz
domain $D$; i.e. for a scalar function $\phi \in H^{1}(D)$, we denote $%
\gamma (w)$ to be the trace mapping from $H^{1}(D)$ to $H^{1/2}(\partial D)$%
. That is, $\gamma (w)=w|_{\partial D},$ for $w\in C^{\infty}(D).$ We will also denote pertinent duality pairings as $(\cdot ,\cdot
)_{X\times X^{\prime }}$.

\section{Preliminaries}

The PDE system given in (\ref{2a})-(\ref%
{IC}) may be associated with an abstract ODE in Hilbert space $\mathbf{H};$ Namely,

\begin{equation}
\left\{ 
\begin{array}{l}
\frac{d}{dt}\Phi (t)=\mathbf{A}\Phi (t)\\
\Phi (0)=\Phi _{0}
\label{ODE}
\end{array}%
\right. 
\end{equation}
where $\Phi (t)=\left[ u(t),h(t),h_{t}(t),w(t),w_{t}(t)\right]
,$ and $\Phi _{0}=[u_{0},h_{0},h_{1},w_{0},w_{1}].$ Here, the operator $\mathbf{A}:D(\mathbf{A})\subset 
\mathbf{H}\rightarrow \mathbf{H}$ is defined by%
\begin{equation}
\mathbf{A}=\left[ 
\begin{array}{ccccc}
\text{div}(\nabla (.)+\nabla ^{T}(.)) & 0 & 0 & 0 & 0 \\ 
0 & 0 & I & 0 & 0 \\ 
-[\nu \cdot (\nabla (.)+\nabla ^{T}(.))]|_{\Gamma _{s}} & \Delta _{\Gamma
_{s}}(.) & 0 & \nu \cdot \sigma (\cdot )|_{\Gamma _{s}} & 0 \\ 
0 & 0 & 0 & 0 & I \\ 
0 & 0 & 0 & \text{div}\sigma (\cdot )-I & 0%
\end{array}%
\right]   \nonumber
\end{equation}%
\begin{equation}
+\left[ 
\begin{array}{ccccc}
\text{-}\nabla \mathcal{P}_{1}(.) & \text{-}\nabla \mathcal{P}_{2}(.) & 0 & \text{-}\nabla
\mathcal{P}_{3}(.) & 0 \\ 
0 & 0 & 0 & 0 & 0 \\ 
\mathcal{P}_{1}(.) \nu  & \mathcal{P}_{2}(.) \nu  & 0 & \mathcal{P}_{3}(.) \nu  & 0 \\ 
0 & 0 & 0 & 0 & 0 \\ 
0 & 0 & 0 & 0 & 0%
\end{array}%
\right].  \label{gener}
\end{equation}%
Here, the ``pressure" operators $\mathcal{P}_{i}$ are as defined below in \eqref{9.6}-\eqref{9.8}. As established in \cite{Ger}, the domain $D(\mathbf{A})$ of the generator $\mathbf{A}$ is characterized as follows: $\left[ u_{0},h_{0},h_{1},w_{0},w_{1}\right] \in D(\mathbf{A}) \Longleftrightarrow $ \\

\noindent \textbf{(A.i)} \text{ \ } $u_{0}\in \bh^{1}(\Omega _{f}),$
%\text{, \ div}%
%(u_{0})=0,\text{ \ \ }u_{0}\cdot \nu |_{\Gamma _{f}}=0;\text{ \ \ \ } \\ 
\text{ \ \  }$h_{1}\in \bh^{1}(\Gamma _{s}),$\text{ \ \  }%
$w_{1}\in \bh^{1}(\Omega _{s});$

\vspace{0.3cm}

\noindent \textbf{(A.ii)} There exists an associated $L^2(\Omega_f)$-function $p_0=p_0(u_0,h_0,w_0)$ such that $$[\text{div}(\nabla u_{0}+\nabla
^{T}u_{0})-\nabla p_{0}]\in L^{2}(\Omega _{f}).$$ Consequently, $p_0$ is harmonic and so  one has the boundary traces\\ \\
\textbf{(a)} $[p_{0}|_{\partial \Omega_f},\frac{\partial p_0}{\partial \nu}|_{\partial \Omega_f}]\in H^{-\frac{1}{2}}(\partial \Omega_f)\times H^{-\frac{3}{2}}(\partial \Omega_f);$\\ \\
\textbf{(b)} $(\nabla u_{0}+\nabla
^{T}u_{0})\cdot \nu \in H^{-\frac{3}{2}}(\partial \Omega _{f}),$

 \vspace{0.3cm}

\noindent \textbf{(A.iii)} \text{div}$\sigma
(w_{0})\in L^{2}(\Omega _{s});$ \text{ \  consequently,\  } $\nu \cdot \sigma
(w_{0}) \in H^{-\frac{1}{2}}(\Gamma_s),$

 \vspace{0.3cm}
 
\noindent \textbf{(A.iv)}  $\Delta _{\Gamma _{s}}(h_{0})+[\nu
\cdot \sigma (w_{0})]|_{\Gamma _{s}}-[(\nabla u_{0}+\nabla ^{T}u_{0})\cdot
\nu ]|_{\Gamma _{s}}+[p_{0} \nu ]|_{\Gamma _{s}}\in L^{2}(\Gamma _{s}),$

\vspace{0.3cm}

\noindent \textbf{(A.v)} $u_{0}|_{\Gamma _{f}}=0,\ \
u_{0}|_{\Gamma _{s}}=h_{1}=w_{1}|_{\Gamma _{s}}$

\vspace{0.3cm}

\noindent We note that the elimination of the pressure variable is very crucial in order to formulate the PDE system (\ref{2a})-(\ref%
{IC}) as an ODE problem given in \eqref{ODE}. For this, we basically apply the divergence operator to the Stokes equation in
(\ref{2a}), and use the fact that $u$ is solenoidal. This gives that the
(pointwise) pressure variable $p(t)$ is harmonic; i.e.,%
\begin{equation}
\Delta p(t)=0\text{ \ \ \ in }\Omega _{f}.  \label{p-1}
\end{equation}

\noindent Subsequently, we multiply (\ref{2.5b}) by $\nu |_{\Gamma _{s}}$ and use the matching velocity
condition in (\ref{2d}) along with system (\ref{2a}) to obtain the following boundary condition for the
pressure variable $p:$%
\begin{equation}
p+\frac{\partial p}{\partial \nu }=\text{div}(\nabla (u)+\nabla
^{T}(u))\cdot \nu |_{\Gamma _{s}}+[(\nabla u+\nabla ^{T}u)\cdot \nu -\Delta
_{\Gamma _{s}}(h)-\nu \cdot \sigma (w)|_{\Gamma _{s}}]\cdot \nu |_{\Gamma
_{s}}. \label{p-2}
\end{equation}%
Also, since $u$ is divergence free, if we take the inner product of both sides of \eqref{2a}$_1,$ with an extension of the normal vector, and subsequently take the trace of this relation on  $\Gamma_f,$ we have $$\frac{\partial p}{\partial \nu}=[\text{div}(\nabla u+\nabla^T u)]\cdot \nu~~on~~ \Gamma_f.$$
Accordingly, the pressure variable $p(t),$ as the solution of (\ref{p-1}%
)-(\ref{p-2}), can formally be written pointwise in time as%
\[
p(t)=\mathcal{P}_{1}(u(t))+\mathcal{P}_{2}(h(t))+\mathcal{P}_{3}(w(t)) \label{9.5}
\]%
where the harmonic functions $\mathcal{P}_{1}(u(t)),$ $\mathcal{P}_{2}(h(t))$, and $\mathcal{P}_{3}(w(t))$
solve the following elliptic BVPs:%

\begin{equation}
\left\{ 
\begin{array}{l}
\Delta \mathcal{P}_{1}(u) =0\text{ \ \ \ \ \   in \ \  }\Omega _{f}   \\
\mathcal{P}_{1}(u)=\text{div}(\nabla
(u)+\nabla ^{T}(u))\cdot \nu |_{\Gamma _{s}}+([(\nabla u+\nabla ^{T}u)]\cdot
\nu )\cdot \nu |_{\Gamma _{s}}\text{\ \ \ \ \ \ on \ \  }\Gamma _{s},\\
\frac{\partial \mathcal{P}_{1}(u)}{\partial \nu } =\text{div}(\nabla
(u)+\nabla ^{T}(u))\cdot \nu |_{\Gamma _{f}}\text{\ \ \ \ \ on \ \  }\Gamma _{f},
\end{array}%
\right. \label{9.6}
\end{equation}
 
\begin{equation}
\left\{ 
\begin{array}{l}
\Delta \mathcal{P}_{2}(h) =0\text{ \ \ \ \ \ \ \ \ \ \ \ \ \ \ \ \ \ \ \  in \ \  }\Omega _{f}  \\
\mathcal{P}_{2}(h)=-\Delta _{\Gamma
_{s}}(h)\cdot \nu |_{\Gamma _{s}}\text{\ \ \ \ \ on \ \  }\Gamma _{s},\\
\frac{\partial \mathcal{P}_{2}(h)}{\partial \nu } =0\text{\ \ \ \ \ \ \ \ \ \ \ \ \ \ \ \ \ \ \ \ \ \  on \ \  }\Gamma _{f},
\end{array}%
\right. \label{9.7}
\end{equation}%
and 
\begin{equation}
\left\{ 
\begin{array}{l}
\Delta \mathcal{P}_{3}(w) =0\text{ \ \ \ \ \ \ \ \ \ \ \ \ \ \ \ \ \ \ \ \ \ \ \ \ \  in \ \  }\Omega _{f}  \\
\mathcal{P}_{3}(w)=-[\nu \cdot \sigma
(w)|_{\Gamma _{s}}]\cdot \nu |_{\Gamma _{s}}\text{\ \ \ \ \ on \ \  }\Gamma _{s}\\
\frac{\partial \mathcal{P}_{3}(w)}{\partial \nu } =0\text{\ \ \ \ \ \ \ \ \ \ \ \ \ \ \ \ \ \ \ \ \ \ \ \ \ \ \ \  on \ \  }\Gamma _{f},
\end{array}%
\right. \label{9.8}
\end{equation}%
The construction of these $\mathcal{P}_{i}$ functions, defined as the solutions to above harmonic equations, allows for the elimination of the pressure term in the system (\ref{2a})-(\ref{IC}). As such, the
pressure-free system can indeed be associated with the abstract ODE \eqref{ODE} in Hilbert
space $\mathbf{H},$ and the associated pressure function $p_0$ in \textbf{(A.ii)} can be identified explicitly, via \begin{equation} p_0=\mathcal{P}_{1}(u_0)+\mathcal{P}_{2}(h_0)+\mathcal{P}_{3}(w_0). \label{11.5} \end{equation}

\noindent (See \cite{Ger}). The wellposedness of the PDE system (\ref{2a})-(%
\ref{IC}), or equivalently the abstract ODE system \eqref{ODE} was proved in \cite{Ger} by means of a nonstandard mixed variational formulation and the use of Lumer Philips Theorem.  For reader's reference, we state here the recent wellposedness result:
\begin{theorem}
\label{well}\cite{Ger} With reference to problem (\ref{2a})-(\ref{IC}), the operator $\mathbf{A}:D(\mathbf{A})\subset \mathbf{H}%
\rightarrow \mathbf{H}$, defined in (\ref{gener}), generates a $%
C_{0}$-semigroup of contractions on $\mathbf{H}$. Consequently, the solution\\ 
$\Phi (t)=\left[ u(t),h(t),h_{t}(t),w(t),w_{t}(t)\right] $ of (\ref{2a})-(\ref{IC}), or
equivalently (\ref{ODE}), is given by 
\[
\Phi (t)=e^{\mathbf{A}t}\Phi _{0}\in C([0,T];\mathbf{H})\text{,}
\]%
where $\Phi _{0}=\left[ u_{0},h_{0},h_{1},w_{0},w_{1}\right] \in \mathbf{H}$.
\end{theorem}
\begin{remark}
Given the form of the adjoint $A^*:D(A^*)\subset \mathbf{H} \rightarrow \mathbf{H},$ in Proposition \ref{form-adj} below, it is readily seen that $\lambda>0$ is also an eigenvalue of $A^*$ with $$Null(\mathbf{A}^*)=Null(\mathbf{A})=\mathbf{N}.$$ Consequently, one has ``invariance of the flow" on $\mathbf{N}^{\bot}.$ That is, if $\Phi_0 \in \mathbf{N}^{\bot},$ then $$e^{\mathbf{A}t}\Phi _{0}\in C([0,T];\mathbf{N}^{\bot}).$$ 

\end{remark}

\section{Main Result: Strong Stability of the Multilayered structure-Stokes fluid PDE System}

This section is devoted to addressing the issue of asymptotic behavior of the
solution whose existence-uniqueness is guaranteed by Theorem \ref{well}. In
this regard, we show that the system given in (\ref{2a})-(%
\ref{IC}) is
strongly stable in the space $\mathbf{N}^{\bot}$ (see \eqref{complement} for definition). Before giving our main result, we state the following Assumption which is crucial in our proof:
\begin{assumption}
\label{assump} Given a fixed $\beta \neq 0,$ suppose that the function $w_0 \in \bh^1(\Omega_s)$ satisfies the following static PDE problem:
\begin{equation*}
\left\{ 
\begin{array}{l}
-\beta^2 w_{0}-\text{div}~\sigma (w_{0})+w_{0}=0\text{ \ \ \ in \ }%
\Omega _{s} \\ 
w_{0}|_{\Gamma _{s}}=0\text{ \ \ \ on \ }\Gamma _{s}\\
\nu \cdot \sigma(w_0)=-c_0 \nu \text{ \ \ \ on \ }\Gamma _{s}.
\end{array}%
\right. 
\end{equation*}%
Then, the solution of this overdetermined problem is $w_0=0,$ and so necessarily $c_0=0.$
\end{assumption}
Such overdetermined eigenvalue boundary value problems also play a part in considering stability properties of other fluid-structure PDE systems. See, e.g. \cite{Av, av-trig, AvalosTriggiani09}. Now, we give our stability result:
\begin{theorem}
\label{main} 
Let Assumption \ref{assump} hold. Then, with reference to the dynamical system (\ref{2a})-(%
\ref{IC}), for any $[\widetilde{u}_{0},\widetilde{h}_{0},\widetilde{h}_{1},\widetilde{w}_{0},\widetilde{w}_{1}]\in 
\mathbf{N}^{\bot},$
the corresponding solution $\left[ u(t),{h}(t),h_{t}(t),w(t),w_{t}(t)\right] \in
C([0,T];\mathbf{N}^{\bot})$ of (\ref{2a})-(%
\ref{IC}) satisfies%
\begin{equation*}
\lim_{t\rightarrow \infty }||\left[ u(t),{h}(t),h_{t}(t),w(t),w_{t}(t)%
\right] || _{\mathbf{N}^{\bot}}=0.
\end{equation*}
\end{theorem}
The domain $D(\mathbf{A})$ of the semigroup generator is not compactly embedded in $\mathbf{H},$ and so classical weak stability approaches are inapplicable here. See e.g., \cite[page 378, Corollary 3.1]{bench}. The proof of Theorem \ref{main} is based on the well known spectral criterion given by W. Arendt and C.J.K. Batty \cite{A-B} (see Appendix, Theorem \ref{sta}). The application of Theorem \ref{sta} on the dynamical system (\ref{2a})-(\ref{IC}) relies on analyzing the spectral properties of the operator $\mathbf{A}$ defined in \eqref{gener}. We will give our analysis in a few steps: 
\vspace{0.3cm}

\noindent\textbf{Zero eigenvalue for the generator $\mathbf{A}$ and explicit characterization of $\mathbf{N}=Null(\mathbf{A})$}
\vspace{0.3cm}

\noindent In order to analyze the long term dynamics of the system (\ref{2a})-(\ref{IC}), we need to avoid steady states so as to reasonably consider the possibility of finite energy solutions tending to the zero state at infinity. For this reason, in the process of applying  Theorem \ref{sta}, we firstly show that zero is an eigenvalue for the operator $\mathbf{A}$. Moreover,  we give an explicit characterization for the corresponding zero eigenspace $\mathbf{N}=Null(\mathbf{A})$ and its orthogonal complement $\mathbf{N}^{\bot}=Null(\mathbf{A})^{\bot}.$

\begin{theorem}
\label{eigen} 
With reference to the PDE system (\ref{2a})-(\ref{IC}) and the corresponding generator $\mathbf{A}$ defined in \eqref{gener}, the point zero is an eigenvalue for $\mathbf{A}.$ Moreover, the following explicit characterizations follow: \begin{equation} \mathbf{N}=Null(\mathbf{A})=span\{\phi\in \mathbf{H}:\phi = \left[\begin{array}{c}0 \\h_0(1) \\0 \\w_0(1) \\0\end{array}\right] \} .\label{null} \end{equation} 
Here, for any scalar $\alpha,$ the pair 
$$[h_0(\alpha),w_0(\alpha)]\in \mathbf{S}=\{[f,g]\in \bh^1(\Gamma_s)\times \bh^1(\Omega_s): f=g|_{\Gamma_s}\}$$ satisfies the variational relation
$$\left\langle \nabla _{\Gamma _{s}}(h_{0}),\nabla _{\Gamma
_{s}}(f) \right\rangle _{\Gamma _{s}}+\left\langle \sigma(w_0),\epsilon(g) \right\rangle _{\Omega _{s}}+\left\langle w_0,g \right\rangle _{\Omega _{s}}=\alpha \left\langle \nu,f \right\rangle _{\Gamma _{s}},~~ \forall~[f,g]\in \mathbf{S}.$$
Also,
 \begin{equation}\mathbf{N}^{\bot}=Null(\mathbf{A})^{\bot}=\{\widetilde{\phi}=\left[\begin{array}{c}\widetilde{u_0} \\\widetilde{h_0}\\ \widetilde{h_1} \\\widetilde{w_0} \\\widetilde{w_1}\end{array}\right]\in \mathbf{H}: \int_{\Gamma_s}( \nu \cdot \widetilde{h_0})~ d\Gamma_s=0\}. \label{comp}\end{equation} 
\end{theorem}

\begin{proof}
\noindent Suppose $\Phi =\left[
u_{0},h_{0},h_{1},w_{0},w_{1}\right] \in D(\mathbf{A})$ is a solution of 
\begin{equation}\mathbf{A}\Phi=0, \label{zero}
\end{equation} 
where $\mathbf{A},$ is as given in \eqref{gener}. With pressure term $p_0$ as defined in \eqref{11.5}, this equation generates the following PDEs:
\begin{equation}
\left\{ 
\begin{array}{l}
\text{div}(\nabla u_{0}+\nabla ^{T}u_{0})-\nabla p_0=0\text{ \ \ \ in \ }\Omega _{f} \\ 
\text{div}(u_0)=0~~~in~~~ \Omega_f\\
h_1=0 \text{ \ \ \ in \ }\Gamma _{s}\\ 
-[\nu \cdot (\nabla u_{0}+\nabla ^{T}u_{0})]|_{\Gamma
_{s}}+\Delta _{\Gamma _{s}}(h_{0})+[\nu \cdot \sigma (w_0)]|_{\Gamma
_{s}}+p_0 \nu =0\text{ \ \ \ in \ }\Gamma_{s}\\
w_1=0\text{ \ \ \ in \ }\Omega _{s}\\
\text{div}\sigma (w_{0})-w_{0}=0\text{ \ \ \ in \ }\Omega _{s}\\
u_0|_{\Gamma_f}=0,~~u_0|_{\Gamma_s}=h_1=0.
\end{array}%
\right.   \label{e1}
\end{equation}
Multiplying \eqref{zero} by $\Phi,$ we take
\[
\left\langle \mathbf{A}\Phi ,\Phi \right\rangle _{\mathbf{H}}=\left\langle 
\text{div}(\nabla (u_{0})+\nabla ^{T}(u_{0})),u_{0}\right\rangle _{\Omega
_{f}}+\left\langle \nabla _{\Gamma _{s}}(h_{1}),\nabla _{\Gamma
_{s}}(h_{0})\right\rangle _{\Gamma _{s}}
\]%
\[
+\left\langle -(\nabla u_{0}+\nabla ^{T}u_{0})\cdot \nu |_{_{\Gamma
_{s}}},h_{1}\right\rangle _{\Gamma _{s}}+\left\langle \Delta _{\Gamma
_{s}}(h_{0}),h_{1}\right\rangle _{\Gamma _{s}}+\left\langle \sigma
(w_{0})\cdot \nu |_{\Gamma _{s}},h_{1}\right\rangle _{\Gamma _{s}}
\]%
\[
+\left\langle \sigma (w_{1}),\epsilon (w_{0})\right\rangle _{\Omega
_{s}}+\left\langle w_{1},w_{0}\right\rangle _{\Omega _{s}}+\left\langle 
\text{div}\sigma (w_{0}),w_{1}\right\rangle _{\Omega _{s}}-\left\langle
w_{0},w_{1}\right\rangle _{\Omega _{s}}
\]%
\[
-\left\langle [\nabla \mathcal{P}_{1}(u_{0})+\nabla \mathcal{P}_{2}(h_{0})+\nabla
\mathcal{P}_{3}(w_{0})],u_{0}\right\rangle _{\Omega _{f}}
\]%
\[
+\left\langle [\mathcal{P}_{1}(u_{0})\cdot \nu +\mathcal{P}_{2}(h_{0})\cdot \nu
+\mathcal{P}_{3}(w_{0})\cdot \nu ],h_{1}\right\rangle _{\Gamma _{s}}=0.
\]%
Applying Green's theorem, using the fact that $u_{0}$ is solenoidal, $u_{0}=0
$ on $\partial \Omega_{f},$ and $h_1=w_1|_{\Gamma_s}=0$ on $\Gamma_s$ we get 
\[
\left\langle \mathbf{A}\Phi ,\Phi \right\rangle _{\mathbf{H}}=-\frac{1}{2}\left\Vert \nabla (u_{0})+\nabla ^{T}(u_{0})\right\Vert ^{2}
+\left\langle \sigma (w_{1}),\epsilon (w_{0})\right\rangle _{\Omega
_{s}}+\left\langle w_{1},w_{0}\right\rangle _{\Omega _{s}}-\left\langle \sigma (w_{0}),\epsilon (w_{1})\right\rangle _{\Omega
_{s}}-\left\langle w_{0},w_{1}\right\rangle _{\Omega _{s}}=0.
\]
or,%
\[
\left\langle \mathbf{A}\Phi ,\Phi \right\rangle _{\mathbf{H}}=-\frac{1}{2}%
\left\Vert \nabla (u_{0})+\nabla ^{T}(u_{0})\right\Vert ^{2}+2i\text{Im}\{%
\left\langle \sigma
(w_{1}),\epsilon (w_{0})\right\rangle _{\Omega _{s}}+\left\langle
w_{1},w_{0}\right\rangle _{\Omega _{s}}\}=0.
\]%
Hence 
\begin{equation}
\text{Re}\left\langle \mathbf{A}\Phi ,\Phi \right\rangle _{\mathbf{H}}=-%
\frac{1}{2}\left\Vert \nabla (u_{0})+\nabla ^{T}(u_{0})\right\Vert ^{2}, \label{relation}
\end{equation}%
which, considering also the boundary condition $u_0|_{\Gamma_f}=0,$ gives us 
\begin{equation}
u_0=0~~ in~~ \Omega_{f}. 
\label{u-0}
\end{equation}
In turn, we have from \eqref{e1}$_1$ that $$p_0=c_0~(constant).$$ Now, define the space $$\mathbf{S}=\{ [f,g] \in \bh^1(\Gamma_s)\times \bh^1(\Omega_s): f=g|_{\Gamma_s}\}.$$ Multiplying \eqref{e1}$_4$ by $f$ and \eqref{e1}$_6$ by $g,$ we have
\[
-\left\langle \nabla _{\Gamma _{s}}(h_{0}),\nabla _{\Gamma
_{s}}(f)\right\rangle _{\Gamma _{s}}+\left\langle \nu\cdot \sigma (w_{0})|_{\Gamma _{s}},f\right\rangle _{\Gamma _{s}}+\left\langle c_0 \nu,f\right\rangle _{\Gamma _{s}}
\]%
\begin{equation}
-\left\langle
\sigma (w_{0})\cdot \nu |_{\Gamma _{s}},g\right\rangle _{\Gamma _{s}}-\left\langle \sigma (w_{0}),\epsilon (g)\right\rangle _{\Omega
_{s}}-\left\langle w_{0},g\right\rangle _{\Omega _{s}}=0. \label{v-1}
\end{equation}
Using the fact that $f=g|_{\Gamma_s},$ we take the variational form in terms of the solution variables $\{h_0,w_0\}:$
\begin{equation}\mathbf{a}([h_{0},w_{0}],[f,g])=\mathbf{F}([f,g]),\text{ \ \ \ }\forall \text{ }[f ,g]\in \textbf{S},   \label{var}\end{equation}
where the bilinear form $\mathbf{a}(.,.):\textbf{S}\times \textbf{S}\rightarrow 
%TCIMACRO{\U{211d} }%
%BeginExpansion
\mathbb{R}
%EndExpansion
$
is defined as 
$$\mathbf{a}([h_{0},w_{0}],[f,g])=\left\langle \nabla _{\Gamma _{s}}(h_{0}),\nabla _{\Gamma
_{s}}(f)\right\rangle _{\Gamma _{s}}+\left\langle \sigma (w_{0}),\epsilon (g)\right\rangle _{\Omega
_{s}}+\left\langle w_{0},g\right\rangle _{\Omega _{s}}
$$
and 
$$\mathbf{F}([f,g])=c_0\left\langle \nu,f\right\rangle _{\Gamma _{s}}.$$
Since it can easily be seen that the bilinear form $\mathbf{a}(.,.)$ is continuous and $\mathbf{S}$-elliptic, the application of Lax Milgram Theorem gives the existence and uniqueness of a solution $[h_0,w_0]\in \mathbf{S}$ to the variational equation \eqref{var}. To conclude the proof of Theorem \ref{eigen}, we must show that the derived solution $\left[
h_{0},w_{0}\right]\in D(\mathbf{A})$ and satisfies the equations in \eqref{e1}. To this end, if we take $g\in D(\Omega_s),$ and $f=0$ in \eqref{v-1} then we have
$$\left\langle -\text{div} \sigma(w_0)+w_0,g\right\rangle =0,~~~\forall~g\in D(\Omega_s),$$ and hence $$-\text{div} \sigma(w_0)+w_0=0,~~~\text{in}~~L^2(\Omega_s).$$
In consequence, we have
\begin{equation}||\sigma(w_0)\cdot \nu||_{H^{-1/2}(\Gamma_s)}\leq C||w_0||_{H^1(\Omega_s)}\leq C|c_0|. \label{i}\end{equation}
In turn: let $\gamma^+_0\in \mathcal{L}(H^{1/2}(\Gamma_s),H^1(\Omega_s))$ be the right inverse of the Dirichlet trace map\\ \\ $\gamma_0: H^1(\Omega_s)\rightarrow H^{1/2}(\Gamma_s).$ Therewith, setting $$[f,g]\equiv [f,\gamma^+_0(f)]$$ in \eqref{v-1} where $f\in \bh^1(\Gamma_s),$ we have
\[
\left\langle \nabla _{\Gamma _{s}}(h_{0}),\nabla _{\Gamma
_{s}}(f)\right\rangle _{\Gamma _{s}}+\left\langle \sigma (w_{0}),\epsilon (\gamma^+_0(f))\right\rangle _{\Omega
_{s}}+\left\langle w_{0},\gamma^+_0(f)\right\rangle _{\Omega _{s}}=c_0\left\langle  \nu,f\right\rangle _{\Gamma _{s}}
\]%
or
\[
\left\langle \nabla _{\Gamma _{s}}(h_{0}),\nabla _{\Gamma
_{s}}(f)\right\rangle _{\Gamma _{s}}+\left\langle \sigma (w_{0}),\epsilon (\gamma^+_0(f))\right\rangle _{\Omega
_{s}}+\left\langle w_{0},\gamma^+_0(f)\right\rangle _{\Omega _{s}}+\left\langle \text{div} \sigma(w_0)-w_0,\gamma^+_0(f)\right\rangle _{\Gamma _{s}}=c_0\left\langle  \nu,f\right\rangle _{\Gamma _{s}}.
\]%
Application of the Green's Theorem gives then
$$\left\langle -\Delta_{\Gamma_s}(h_{0})+\nu\cdot \sigma(w_0)\right\rangle _{\Gamma _{s}}=c_0\left\langle \nu,f\right\rangle _{\Gamma _{s}}=0$$
or 
\begin{equation}-\Delta_{\Gamma_s}(h_{0})+\nu\cdot \sigma(w_0)=c_0 \nu~~~in~~~ L^2(\Gamma_s).\label{ii} \end{equation}
In sum, we have that the obtained vector $\phi = \left[\begin{array}{c}0 \\h_0(1) \\0 \\w_0(1) \\0\end{array}\right]\in D(\mathbf{A})$ solves $$\mathbf{A}\phi =0,$$ and is indeed the zero eigenvector with the corresponding zero eigenspace $\mathbf{N}=Null\mathbf(A)$ could be characterized as in \eqref{null}. Subsequently, recalling the inner product introduced in \eqref{Hilbert}, and using the relations \eqref{i}-\eqref{ii}, the orthogonal complement $\mathbf{N}^{\bot}$ follows as in \eqref{comp}. This completes the proof of Theorem \eqref{eigen}.
\end{proof}
\\

\noindent The proof of Theorem \eqref{main} will rely on the ultimate application of the spectral criterion of W. Arendt and C. Batty for strong decay (see Appendix, Theorem \ref{sta}). This  will entail the elimination of all three parts of the spectrum of the generator $\mathbf{A}$ from the imaginary axis: In this connection, we now proceed with the analysis of the point spectrum $\sigma_{p}(\mathbf{A}).$
% \vspace{0.5cm}
 
\begin{lemma} 
\label{point} 
Let Assumption \ref{assump} hold. With reference to the PDE system (\ref{2a})-(\ref{IC}) and the corresponding generator $\mathbf{A}$ defined in \eqref{gener}, given for $\beta \neq 0,$ $i\beta \notin \sigma_{p}(\mathbf{A}).$
\end{lemma}

\begin{proof} Suppose $\Phi =\left[
u_{0},h_{0},h_{1},w_{0},w_{1}\right] \in D(\mathbf{A})$ satisfies the relation 
\begin{equation}
(i\beta I-\mathbf{A})\Phi =0. \label{p1}
\end{equation}%
In PDE terms, we then have
\begin{equation}
\left\{ 
\begin{array}{l}
i\beta u_{0}-\text{div}(\nabla u_{0}+\nabla ^{T}u_{0})+\nabla p_0=0\text{ \ \ \ in \ }\Omega _{f} \\ 
\text{div}(u_{0})=0\text{ \ \ \ \ \ \ in \ }\Omega _{f} \\ 
u_{0}|_{\Gamma _{f}}=0\text{ \ \ \ on \ }\Gamma _{f};%
\end{array}%
\right.   \label{s1}
\end{equation}

\begin{equation}
\left\{ 
\begin{array}{c}
\begin{array}{l}
i\beta h_{0}-h_{1}=0\text{ \ \ in }\Gamma _{s}%
\end{array}
\\ 
i\beta h_{1}+[\nu \cdot (\nabla u_{0}+\nabla ^{T}u_{0})]|_{\Gamma
_{s}}-\Delta _{\Gamma _{s}}(h_{0})-[\nu \cdot \sigma (w_0)]|_{\Gamma
_{s}}-p_0 \nu =0\text{ \ \ \ in\ \ }\Gamma _{s},\text{ \ \ }%
\end{array}%
\right.   \label{s2}
\end{equation}%
\begin{equation}
\left\{ 
\begin{array}{l}
i\beta w_{0}-w_{1}=0\text{ \ \ in \ \ \ }\Omega _{s} \\ 
i\beta w_{1}-\text{div}\sigma (w_{0})+w_{0}=0\text{ \ \ \ in \ }%
\Omega _{s} \\ 
w_{1}|_{\Gamma _{s}}=h_{1}=u_{0}|_{\Gamma _{s}}\text{ \ \ \ on \ }\Gamma _{s}%
\end{array}%
\right.   \label{s3}
\end{equation}%
with $p_0$ being the associated pressure of the PDE system. Therewith, the usual energy method and the relation \eqref{relation} gives
\begin{equation}
\text{Re}\left\langle (i\beta I- \mathbf{A})\Phi ,\Phi \right\rangle _{\mathbf{H}}=-\text{Re}\left\langle \mathbf{A}\Phi ,\Phi \right\rangle _{\mathbf{H}}=%
\frac{1}{2}\left\Vert \nabla (u_{0})+\nabla ^{T}(u_{0})\right\Vert ^{2}= 0 \label{relation1}
\end{equation}%
which, considering also the boundary condition $u_0|_{\Gamma_f}=0,$ gives 
\begin{equation}
u_0=0~~ in~~ \Omega_{f}. 
\label{u0}
\end{equation}
Also, taking into account \eqref{u0} in \eqref{s3}$_3$, we have
\begin{equation}
w_1|_{\Gamma_s}=h_1=0~~ on~~ \Gamma_{s}. 
\label{h1}
\end{equation} 
Let $p_0=q_0+c_0,$ where $q_0\in \hat{L}^{2}(\Omega _{f}),$ where \begin{equation} \hat{L}^{2}(\Omega _{f})=\{f\in L^2(\Omega_f): \int_{\Omega_f} f d\Omega_f=0\}, \label{lhat}\end{equation} and $c_0$ is a constant. Then by \eqref{u0} and \eqref{s1}$_1,$   
\begin{equation}
q_0=0~~ in~~ \Omega_{f}. 
\label{q0}
\end{equation} 
Also, from \eqref{h1} and \eqref{s2}$_1$, we have
\begin{equation}
h_0=0~~ in~~ \Gamma_{s}. 
\label{h0}
\end{equation} 
Now, if we consider \eqref{h1} and \eqref{h0} in the thin layer equation \eqref{s2}$_2,$ as well as \eqref{s3}$_1$ in \eqref{s3}$_2,$ we get that if $\Phi =\left[
u_{0},h_{0},h_{1},w_{0},w_{1}\right] \in D(\mathbf{A})$ is an eigenfunction corresponding to eigenvalue $i\beta~~(\beta \neq 0),$ then $w_0$ solves the following overdetermined eigenvalue problem:
\begin{equation}
\left\{ 
\begin{array}{l}
-\beta^2 w_{0}-\text{div}\sigma (w_{0})+w_{0}=0\text{ \ \ \ in \ }%
\Omega _{s} \\ 
w_{0}|_{\Gamma _{s}}=0\text{ \ \ \ on \ }\Gamma _{s}\\
\nu \cdot \sigma(w_0)=-c_0 \nu \text{ \ \ \ on \ }\Gamma _{s}.
\end{array}%
\right.   \label{overdet}
\end{equation}%
Exploiting Assumption \ref{assump} for the problem \eqref{overdet} gives that $w_0=0$ and $c_0=0,$ which then yield that $ \sigma_{p}(\mathbf{A})\bigcap i \mathbb{R}=\emptyset.$ This completes the proof of Lemma \ref{point}.
\end{proof}
\\

\noindent In our second step, we continue with the analysis of the residual spectrum $\sigma_r(\mathbf{A}):$

\begin{lemma} 
\label{res} 
Let Assumption \ref{assump} hold. With reference to the PDE system (\ref{2a})-(\ref{IC}) and the corresponding generator $\mathbf{A}$ defined in \eqref{gener}, given $\beta \neq 0,$ $i\beta \notin \sigma_{r}(\mathbf{A}).$
\end{lemma}
\vspace{0.5cm
}
\textbf{Proof.} The proof of Lemma \ref{res} relies on the wellknown fact that ``For any closed, densely defined operator $\mathbf{A}$, if $\lambda \in \sigma_r(\mathbf{A})$ then $\overline{\lambda}\in \sigma_p(\mathbf{A}^*)." $ (See e.g., \cite[page 127]{F}). Accordingly, we firstly give a representation of the adjoint operator $\mathbf{A}^*: D(\mathbf{A}^*)\rightarrow \mathbf{H}.$ In fact, a standard computation yields:
\begin{proposition} \label{form-adj}
For the generator operator $\mathbf{A}$ defined in \eqref{gener}, the Hilbert space adjoint\\ $\mathbf{A}^*: D(\mathbf{A}^*)\rightarrow \mathbf{H}$ is given by
\begin{equation}
\mathbf{A}^*=\left[ 
\begin{array}{ccccc}
\text{div}(\nabla (.)+\nabla ^{T}(.)) & 0 & 0 & 0 & 0 \\ 
0 & 0 & -I & 0 & 0 \\ 
-[\nu \cdot (\nabla (.)+\nabla ^{T}(.))]|_{\Gamma _{s}} &- \Delta _{\Gamma
_{s}}(.) & 0 & -\nu \cdot \sigma (\cdot )|_{\Gamma _{s}} & 0 \\ 
0 & 0 & 0 & 0 & -I \\ 
0 & 0 & 0 & -\text{div}~\sigma (\cdot )+I & 0%
\end{array}%
\right]   \nonumber
\end{equation}%
\begin{equation}
+\left[ 
\begin{array}{ccccc}
\text{-}\nabla \mathcal{P}_{1}(.) & \nabla \mathcal{P}_{2}(.) & 0 & \nabla
\mathcal{P}_{3}(.) & 0 \\ 
0 & 0 & 0 & 0 & 0 \\ 
\mathcal{P}_{1}(.) \nu  & -\mathcal{P}_{2}(.) \nu  & 0 & -\mathcal{P}_{3}(.) \nu  & 0 \\ 
0 & 0 & 0 & 0 & 0 \\ 
0 & 0 & 0 & 0 & 0%
\end{array}%
\right].  \label{adj}
\end{equation}%
The domain $D(\mathbf{A}^*)$ of $\mathbf{A}^*$ is characterized as follows: $\left[ \widetilde{u}_{0},\widetilde{h}_{0},\widetilde{h}_{1},\widetilde{w}_{0},\widetilde{w}_{1}\right] \in D(\mathbf{A}^*) \Longleftrightarrow $ \\

\noindent \textbf{(A.i)} \text{ \ } $\widetilde{u}_{0}\in \bh^{1}(\Omega _{f}),$
%\text{, \ div}%
%(u_{0})=0,\text{ \ \ }u_{0}\cdot \nu |_{\Gamma _{f}}=0;\text{ \ \ \ } \\ 
\text{ \ \  }$\widetilde{h}_{1}\in \bh^{1}(\Gamma _{s}),$\text{ \ \  }%
$\widetilde{w}_{1}\in \bh^{1}(\Omega _{s}),$

\vspace{0.3cm}

\noindent \textbf{(A.ii)} There exists an associated $L^2(\Omega_f)$-function $\widetilde{p}_0=\widetilde{p}_0(\widetilde{u}_0,\widetilde{h}_0,\widetilde{w}_0)$ such that $$[\text{div}(\nabla \widetilde{u}_{0}+\nabla
^{T}\widetilde{u}_{0})-\nabla \widetilde{p}_{0}]\in L^{2}(\Omega _{f}).$$ Consequently, $\widetilde{p}_0$ is harmonic and so \\ \\
\textbf{(a)} $[\widetilde{p}_{0}|_{\partial \Omega_f},\frac{\partial \widetilde{p}_0}{\partial \nu}|_{\partial \Omega_f}]\in H^{-\frac{1}{2}}(\partial \Omega_f)\times H^{-\frac{3}{2}}(\partial \Omega_f);$\\
\textbf{(b)} $(\nabla \widetilde{u}_{0}+\nabla
^{T}\widetilde{u}_{0})\cdot \nu \in H^{-\frac{3}{2}}(\partial \Omega _{f}),$

 \vspace{0.3cm}

\noindent \textbf{(A.iii)} \text{div}$\sigma
(\widetilde{w}_{0})\in L^{2}(\Omega _{s});$ \text{ \  consequently,\  } $\nu \cdot \sigma
(\widetilde{w}_{0}) \in H^{-\frac{1}{2}}(\Gamma_s),$

 \vspace{0.3cm}
 
\noindent \textbf{(A.iv)}  $-\Delta _{\Gamma _{s}}(\widetilde{h}_{0})-[\nu
\cdot \sigma (\widetilde{w}_{0})]|_{\Gamma _{s}}-[(\nabla \widetilde{u}_{0}+\nabla ^{T}\widetilde{u}_{0})\cdot
\nu ]|_{\Gamma _{s}}+[\widetilde{p}_{0} \nu ]|_{\Gamma _{s}}\in L^{2}(\Gamma _{s}),$

\vspace{0.3cm}

\noindent \textbf{(A.v)} $\widetilde{u}_{0}|_{\Gamma _{f}}=0,\ \
\widetilde{u}_{0}|_{\Gamma _{s}}=\widetilde{h}_{1}=\widetilde{w}_{1}|_{\Gamma _{s}}$

\end{proposition}
It is readily discerned that $Null(\mathbf{A})=Null(\mathbf{A}^*).$ In turn, $\lambda=0$ is an eigenvalue of $A^*.$ Moreover, under Assumption \ref{assump}, $i\beta~~(\beta \neq 0)$ is not an eigenvalue for the adjoint operator $\mathbf{A}^*$ and hence it is not in the residual spectrum of $\mathbf{A}.$ This completes the proof of Lemma \ref{res}.  
\\

\noindent Now, we continue with analyzing the continuous spectrum of the generator $\mathbf{A}.$ In this connection, we give the following Lemma: 

\begin{lemma}
\label{cont} Let Assumption \ref{assump} hold. With reference to the PDE system (\ref{2a})-(\ref{IC}) and the corresponding generator $\mathbf{A}$ defined in \eqref{gener}, for a given $\beta \neq 0,$ $i\beta \notin \sigma_{c}(\mathbf{A}).$\end{lemma}

\begin{proof}
Our proof is based on a contradiction argument. To start with, asssume that 
 $i\beta $ ($\beta \neq 0$) is in the continuous spectrum $\sigma_c(\mathbf{A})$ of $\mathbf{A}.$ Since $\sigma_c(\mathbf{A}) \subset \sigma_{app}(\mathbf{A})$ (approximate spectrum) (see e.g., \cite[page 128]{F}) then there exist sequences%
\begin{equation}
\{\Phi _{n}\}=\left\{ \left[ 
\begin{array}{c}
u_{0n} \\ 
h_{0n} \\ 
h_{1n} \\ 
w_{0n} \\ 
w_{1n}%
\end{array}%
\right] \right\} \subseteq D(\mathbf{A});\text{ \ \ \ \ \ \ \ }\{(i\beta I-%
\mathbf{A)}\Phi _{n}\}\equiv \{\Phi^* _{n}\}=\left\{ \left[ 
\begin{array}{c}
u_{0n}^{\ast } \\ 
h_{0n}^{\ast } \\ 
h_{1n}^{\ast } \\ 
w_{0n}^{\ast } \\ 
w_{1n}^{\ast }%
\end{array}%
\right] \right\} \subseteq \mathbf{H}\text{\ ,}  \label{cont1}
\end{equation}%
which satisfy for $n=1,2,...,$%
\begin{equation}
\left\Vert \Phi _{n}\right\Vert _{\mathbf{H}}=1,\text{ and  \ \ \ }\left\Vert
(i\beta I-\mathbf{A)}\Phi _{n}\right\Vert _{\mathbf{H}}<\frac{1}{n}.
\label{cont2}
\end{equation}%
In PDE terms, each $\Phi _{n}$ solves the following static system, where again $p_{n}$ is given via \eqref{11.5}: 
\begin{equation}
\left\{ 
\begin{array}{l}
i\beta u_{0n}-\text{div}(\nabla u_{0n}+\nabla ^{T}u_{0n})+\nabla p_{n}=u_{0n}^{\ast
}\text{ \ \ \ in \ }\Omega _{f} \\ 
\text{div}(u_{0n})=0\text{ \ \ \ \ \ \ in \ }\Omega _{f} \\ 
u_{0n}|_{\Gamma _{f}}=0\text{ \ \ \ on \ }\Gamma _{f};%
\end{array}%
\right.   \label{s1-cont}
\end{equation}

\begin{equation}
\left\{ 
\begin{array}{c}
\begin{array}{l}
i\beta h_{0n}-h_{1n}=h_{0n}^{\ast }\text{ \ \ in }\Gamma _{s}%
\end{array}
\\ 
i\beta h_{1n}+[\nu \cdot (\nabla u_{0n}+\nabla ^{T}u_{0n})]|_{\Gamma
_{s}}-\Delta _{\Gamma _{s}}(h_{0n})-[\nu \cdot \sigma (w_{0n})]|_{\Gamma
_{s}}-p_{n} \nu =h_{1n}^{\ast }\text{ \ \ \ in\ \ }\Gamma _{s},\text{ \ \ }%
\end{array}%
\right.   \label{s2-cont}
\end{equation}%
\begin{equation}
\left\{ 
\begin{array}{l}
i\beta w_{0n}-w_{1n}=w_{0n}^{\ast }\text{ \ \ in \ \ \ }\Omega _{s} \\ 
i\beta w_{1n}-\text{div}\sigma (w_{0n})+w_{0n}=w_{1n}^{\ast }\text{ \ \ \ in \ }%
\Omega _{s} \\ 
w_{1n}|_{\Gamma _{s}}=h_{1n}=u_{0n}|_{\Gamma _{s}}\text{ \ \ \ on \ }\Gamma _{s}%
\end{array}%
\right.   \label{s3-cont}
\end{equation}%
We proceed here stepwise:
\vspace{0.3cm}

\noindent \textbf{\underline{Step I:} Estimates for the fluid variable $\{u_{0n}\}$ and thin-layer solution variables $\{h_{0n}, h_{1n}\}$}
\vspace{0.2cm}

\noindent To start, via the integration by parts, dissipativity relation \eqref{relation}, and \eqref{cont2} we have
 
\begin{equation}
\text{Re}\left\langle (\Phi^*_{n} ,\Phi_{n} \right\rangle _{\mathbf{H}}=\text{Re}\left\langle (i\beta I- \mathbf{A})\Phi_{n} ,\Phi_{n} \right\rangle _{\mathbf{H}}=\frac{1}{2}\left\Vert \nabla (u_{0n})+\nabla ^{T}(u_{0n})\right\Vert ^{2}=\mathcal{O}(\frac{1}{n}), \label{cont3}
\end{equation}%
and hence
\begin{equation}
||u_{0n}||_{H^1(\Omega_f)}\rightarrow 0. \label{cont3.5}
\end{equation} 
\vspace{0.1cm}

\noindent In turn, we have by \eqref{s3-cont}, \eqref{cont3} and Sobolev Trace Theory
\begin{equation}
u_{0n}|_{\Gamma_s}=w_{1n}|_{\Gamma_s}=h_{1n}=\mathcal{O}(\frac{1}{n}). \label{cont4}
\end{equation} 
Let $p_{n}=q_{n}+c_{n},$ where $q_n\in \hat{L}^{2}(\Omega _{f})$ (as defined in \eqref{lhat}), and $c_n$ is a constant. Then, by the Stokes Theory \cite{temam}, $\{u_{0n},q_n\}\in \bh^1(\Omega _{f})\times  \hat{L}^{2}(\Omega _{f})$ uniquely solve
\begin{equation}
\left\{ 
\begin{array}{l}
-\text{div}(\nabla u_{0n}+\nabla ^{T}u_{0n})+\nabla
q_{n}=-i\beta u_{0n}+u^{\ast }_{0n}\text{\ \ \ in \ }\Omega _{f} \\ 
\text{div}(u_{0n})=0\text{\ \ \ \ in \ }\Omega _{f} \\ 
u_{0n}|_{ \Gamma_{s}}=u_{0n}|_{ \Gamma_{s}}\text{ \ \ \ on \ } \Gamma_{s}\\
u_{0n}|_{\Gamma_{f}}=0\text{ \ \ \ on \ } \Gamma_{f},%
\end{array}%
\right.   \label{cont5}
\end{equation}
and so the following estimate holds:
\begin{equation}
||q_n||_{L^2(\Omega_f)}\leq C[||\Phi^*_{n}||_{\mathbf{H}}+||u_{0n}||_{H^1(\Omega_f)}]=\mathcal{O}(\frac{1}{n}). \label{cont6}
\end{equation}
Subsequently, an energy method yields that
\begin{equation}
\{\nu \cdot (\nabla u_{0n}+\nabla ^{T}u_{0n})-q_n \cdot \nu \}\in H^{-1/2}(\partial \Omega_f)
\end{equation}
with 
\begin{equation}
||\nu \cdot (\nabla u_{0n}+\nabla ^{T}u_{0n})-q_n \cdot \nu||_{H^{-1/2}(\partial \Omega_f)}\leq C[||u_{0n}^{\ast
}-i\beta u_{0n}||_{\Omega_f}+||u_{0n}||_{H^1(\Omega_f)}]=\mathcal{O}(\frac{1}{n}), \label{cont7}
\end{equation}
after using \eqref{cont2} and \eqref{cont3}. Moreover, since by the thin-layer resolvent relation in \eqref{s2-cont}$_1$ and the boundary condition \eqref{s3-cont}$_3$ 
\begin{equation}
h_{0n}=-\frac{i}{\beta}h_{1n}-\frac{i}{\beta}h_{0n}^{\ast}=-\frac{i}{\beta}u_{0n}-\frac{i}{\beta}h_{0n}^{\ast}.
\end{equation}
Using again \eqref{cont2} and \eqref{cont3}, we get 
\begin{equation}
||h_{0n}||_{H^{1/2}(\Gamma_s)}=\mathcal{O}(\frac{1}{n}). \label{cont8}
\end{equation}
\vspace{0.1cm}

\noindent \textbf{\underline{Step II: }Estimate for the term $\{\nu \cdot \sigma(w_{0n})\}$}
\vspace{0.3cm}

\noindent We start by invoking the ``Dirichlet" map $%
D_{s}:L^{2}(\Gamma _{s})\rightarrow L^{2}(\Omega _{s})$\ via%
\begin{equation*}
D_{s}g=f\Longleftrightarrow \left\{ 
\begin{array}{c}
\text{div}~\sigma (f)=0\text{ \ in \ }\Omega _{s} \\ 
f|_{\Gamma _{s}}=g\text{ \ on \ }\Gamma _{s}.%
\end{array}%
\right. 
\end{equation*}%
We know by the Lax-Milgram Theorem%
\begin{equation}
D_{s}\in \mathcal{L}(H^{\frac{1}{2}}(\Gamma _{s}),H^{1}(\Omega _{s})).
\label{cont9}
\end{equation}%
Now, with the Dirichlet map $D_s$ in hand, we make the change of variable
\begin{equation}
z_{n}\equiv w_{0n}+\frac{i}{\beta}D_{s}[u_{n}|_{\Gamma _{s}}+w_{0n}^{\ast }|_{\Gamma _{s}}].  \label{cont10}
\end{equation}%
Considering the resolvent relations in \eqref{s3-cont}$_1$-\eqref{s3-cont}$_2,$ $z_n$ then solves the following boundary value problem (BVP): 
\begin{equation*}
-\beta ^{2}z_{n}-\text{div}\sigma (z_{n})=F_{\beta}\text{ \ \ in \ \ }\Omega _{s}
\end{equation*}
\begin{equation}
\hspace{2.2cm}z_{n}|_{\Gamma _{s}}=0\text{ \ \ \ on \ \ }\Gamma _{s} \label{g-a}
\end{equation}
where 
\begin{equation}
\hspace{1cm}F_{\beta}= w_{1n}^{\ast }+i\beta w_{0n}^{\ast }-w_{0n}-i\beta D_{s}[u_{n}|_{\Gamma _{s}}+w_{0n}^{\ast }|_{\Gamma _{s}}]\label{g-b}
\end{equation}
Since this BVP has homogeneous boundary data then we have the estimate-- see e.g., \cite[page 296, Theorem 6.3-6]{ciarlet}:
\begin{equation}
||z_n||_{H^2{(\Omega_s)}} \leq ||w_{1n}^{\ast }+i\beta w_{0n}^{\ast }-w_{0n}-i\beta D_{s}[u_{n}|_{\Gamma _{s}}+w_{0n}^{\ast }|_{\Gamma _{s}}]||_{\Omega_s}\leq C_{1,\beta}\label{g-c}
\end{equation}
after using \eqref{cont2}. Consequently, there is the trace mapping-- see, e.g. \cite{Necas},
\begin{equation}
||\frac{\partial z_n}{\partial \nu}||_{\Gamma_s}\leq C||z_n||_{H^2{(\Omega_s)}} \leq C_{2,\beta}\label{g-d}
\end{equation}
(again, after using \eqref{g-c}.) With this trace estimate in hand, we invoke the following known expression for $\{\sigma(z_n)\cdot \nu|_{\Gamma_s}\}$ in terms of the normal and tangential derivatives (see \cite[page 18, Proposition A.1]{GR}):
\begin{equation*}
\sigma(z_n)\cdot \nu=\lambda \Big[ \frac{\partial z_n}{\partial \nu}\cdot \nu+\frac{\partial z_n}{\partial \tau}\cdot \tau+\frac{\partial z_n}{\partial e}\cdot e \Big]\nu+2\mu \frac{\partial z_n}{\partial \nu}
\end{equation*}
\begin{equation}
\hspace{2cm}+\mu  \Big[ \frac{\partial z_n}{\partial \tau}\cdot \nu-\frac{\partial z_n}{\partial \nu}\cdot \tau \Big] \tau+\mu\Big[ \frac{\partial z_n}{\partial e}\cdot \nu-\frac{\partial z_n}{\partial \nu}\cdot e\Big]e \label{g-e}
\end{equation}
Here, unit (tangent) vectors $\{e,\tau\}$ and $\nu$ constitute an orthonormal system on $\mathbb{R}^3.$ Since $z_n=0$ on $\Gamma_s$ then \eqref{g-e} is simplified to 
\begin{equation}
\sigma(z_n)\cdot \nu=\lambda ( \frac{\partial z_n}{\partial \nu}\cdot \nu)+2\mu \frac{\partial z_n}{\partial \nu}-\mu(\frac{\partial z_n}{\partial \nu}\cdot \tau) \tau-\mu(\frac{\partial z_n}{\partial \nu}\cdot e) e. \label{g-f}
\end{equation}
Applying the estimate \eqref{g-d} to the RHS of \eqref{g-f} now yields
\begin{equation}
||\sigma(z_n)\cdot \nu ||_{\Gamma_s} \leq C_{3,\beta}. \label{g-g}
\end{equation}

\noindent Moreover, an integration by parts yields the inference that
\begin{equation*}
\nu \cdot \sigma(D_s(\cdot)) \in \mathcal{L}(H^{1/2}(\Gamma_s),H^{-1/2}(\Gamma_s)).
\end{equation*}
Combining this boundedness with \eqref{cont2}, and \eqref{cont3} we have
\begin{equation}
||\nu \cdot \sigma(D_{s}[u_{n}|_{\Gamma _{s}}+w_{0n}^{\ast }|_{\Gamma _{s}}])||_{H^{-1/2}(\Gamma_s)} =\mathcal{O}(\frac{1}{n}). \label{g-h}
\end{equation}
Applying now \eqref{g-g}-\eqref{g-h} to the relation \eqref{cont10}, we get
\begin{equation}
||\nu \cdot \sigma(w_{0n})||_{H^{-1/2}(\Gamma_s)} \leq C. \label{21}
\end{equation}
\vspace{0.3cm}

\noindent \textbf{\underline{Step III: }Estimate for the term $\{c_n\}$}
\vspace{0.3cm}

\noindent At this step, we recall the definition of the pressure term $p_n=q_n+c_n,$ and read off the equation \eqref{s2-cont}$_2$ to have
\begin{equation*}
||-c_n\nu||_{H^{-1}(\Gamma_s)} =||-i\beta h_{1n}+h_{1n}^{\ast }-[\nu \cdot (\nabla u_{0n}+\nabla ^{T}u_{0n})]|_{\Gamma
_{s}}+\Delta _{\Gamma _{s}}(h_{0n})+[\nu \cdot \sigma (w_0n)]|_{\Gamma
_{s}}+q_{n} \nu ||_{H^{-1}(\Gamma_s)},
\end{equation*}
whence we get via \eqref{cont2}, \eqref{cont6}, \eqref{cont7}, and \eqref{21}
\begin{equation}
|c_n| \leq C, \label{22}
\end{equation}
where again $C>0$ depends on $\lambda, \beta, \mu.$
\vspace{0.3cm}

\noindent \textbf{\underline{Step IV: }Estimate for $\{\nabla_{\Gamma_s}h_{0n}\}$}
\vspace{0.3cm}

\noindent We multiply both sides of the equation \eqref{s2-cont}$_2$ by $h_{0n},$ integrate in space, and then use the integration by parts to take
\begin{equation*}
||\nabla_{\Gamma_s}h_{0n}||^2_{\Gamma_s}=-i\beta \left\langle h_{1n},h_{0n}\right\rangle_{\Gamma_s}+ \left\langle h_{1n}^{\ast },h_{0n}\right\rangle_{\Gamma_s}- \left\langle[\nu \cdot (\nabla u_{0n}+\nabla ^{T}u_{0n})]|_{\Gamma
_{s}},h_{0n}\right\rangle_{\Gamma_s}
\end{equation*}
\begin{equation*}
\hspace{-2.5cm}+ \left\langle[\nu \cdot \sigma (w_{0n})]|_{\Gamma
_{s}},h_{0n}\right\rangle_{\Gamma_s}+ \left\langle p_{n} \nu,h_{0n}\right\rangle_{\Gamma_s}.
\end{equation*}
Invoking again \eqref{cont2}, \eqref{cont3}, \eqref{cont4}, \eqref{cont6}, \eqref{21}-\eqref{22}, and moreover using the resolvent relation \eqref{s2-cont}$_{1}$ for the second term of the RHS of the last equality, we have
\begin{equation}
||\nabla_{\Gamma_s}h_{0n}||^2_{\Gamma_s}=\mathcal{O}(\frac{1}{n}). \label{22.5}
\end{equation}
Now, collecting all the estimates obtained in \eqref{cont3.5}, \eqref{cont4}, \eqref{cont6}, and \eqref{22.5} we get the following convergences:
\begin{equation}
\left\{ 
\begin{array}{l}
u_{0n}\rightarrow 0~~in~~\bh^1(\Omega_f) \\ 
q_n\rightarrow 0~~in~~L^2(\Omega _{f}) \\ 
q_n|_{\Gamma_s}\rightarrow 0~~in~~H^{-1/2}(\Gamma_s)\\
w_{0n}|_{\Gamma_{s}}\rightarrow 0~~in~~\bh^{1/2}(\Gamma_s)\\
h_{0n}\rightarrow 0~~in~~\bh^{1}(\Gamma_s)\\
h_{1n}\rightarrow 0~~in~~\bh^{1/2}(\Gamma_s).
\end{array}%
\right.   \label{23}
\end{equation}
Moreover, from \eqref{g-g}, the sequence $\{\nu \cdot \sigma(z_n)\}$ has a weakly convergent subsequence (still denoted as itself) such that $\{\nu \cdot \sigma(z_n)\}$ converges (weakly) in $L^2(\Gamma_s) $ and since $L^2(\Gamma_s)\subset_{>} H^{-1}(\Gamma_s) $ is compact we have then $\{\nu \cdot \sigma(z_n)\}$ converges strongly in $H^{-1}(\Gamma_s). $ \\

\noindent Lastly, from \eqref{22}, we have that $\{c_n\}$ converges strongly to $c^*,$ say. If we recall the definition of $z_n$ given in \eqref{cont10} and invoke the resolvent relation \eqref{s2-cont}$_{2}$ together with the limits in \eqref{23}, we have 
\begin{equation}
lim_{n\rightarrow \infty} \nu \cdot\sigma(z_n)=-c^*\nu. \label{24}
\end{equation}

\noindent To finish our proof, we consider again the BVP given in \eqref{g-a}: Initially, the relations \eqref{cont2} and \eqref{cont10} provide for the weak convergence 
\begin{equation*}
 z_n \xrightarrow{w} z,~~(say)~~in~~\bh^1(\Omega_s).
\end{equation*}
Hence, if we pass to the limit in \eqref{g-a} when $n\rightarrow \infty,$ then we get that $z\in \bh^1(\Omega_s)$ satisfies the following problem:
\begin{equation*}
-\beta ^{2}\left\langle z,\Psi \right\rangle_{\Omega_s}+\left\langle \sigma (z), \epsilon(\Psi)\right\rangle_{\Omega_s}+\left\langle c^*\nu,\Psi\right\rangle_{\Omega_s}+\left\langle z, \Psi\right\rangle_{\Omega_s}=0,~~\forall \Psi \in \bh^1(\Omega_s).
\end{equation*}
That is, $\{-\beta^2, z\}$ solves the following overdetermined eigenvalue problem:
\begin{equation*}
-\beta ^{2}z-\text{div}\sigma (z)+z=0\text{ \ \ in \ \ }\Omega _{s} 
\end{equation*}
\begin{equation*}
\hspace{3cm}z=0\text{ \ \ on \ \ }\Gamma _{s} 
\end{equation*}
\begin{equation}
\hspace{3.5cm}\frac{\partial z}{\partial \nu}=-c^*\nu\text{ \ \ on \ \ }\Gamma _{s}. \label{25}
\end{equation}
Now, under Assumption \ref{assump}, the only solution to problem \eqref{25} is $z=0$ for $c^*=0.$ Then from \eqref{cont10}, \eqref{cont3.5}, and \eqref{cont2}, we have $w_{0n}\rightarrow 0.$ This convergence, and those in \eqref{23}, contradicts the assumption that
\begin{equation}
||\Phi_n||_{\mathbf{H}}=1.
\end{equation}
As a result, for any $\beta \neq 0,$ $i\beta \notin \sigma_{c}(\mathbf{A}),$ and this completes the proof of Lemma \ref{cont}.
\end{proof}
\\ \\
\noindent To proceed with the proof of Theorem \ref{main}, we recall by Theorem \ref{eigen} that zero is an eigenvalue for the generator $\mathbf{A}.$ We show in fact that  $\lambda=0$ is in the resolvent set of $\mathbf{A}|_{N^{\bot}}:$
\begin{lemma} 
\label{resolv} 
$\lambda=0$ is in the resolvent set $\rho(\mathbf{A}|_{N^{\bot}})$ of $\mathbf{A}|_{N^{\bot}}: D(\mathbf{A}|_{N^{\bot}})\rightarrow N^{\bot}.$ That is, $$\Big(\mathbf{A}|_{N^{\bot}}\Big)^{-1}\in \mathcal{L}(N^{\bot}).$$
\end{lemma}
\begin{proof}
As before, we use the denotations $$\Phi =\left[
u_{0},h_{0},h_{1},w_{0},w_{1}\right] \in D(\mathbf{A})\cap N^{\bot},~~~\Phi^* =\left[
u^*_{0},h^*_{0},h^*_{1},w^*_{0},w^*_{1}\right] \in N^{\bot},$$ and consider solving the following relation
\begin{equation}
\mathbf{A}\Phi=\Phi^*. \label{star}
\end{equation} 
Then, in PDE terms, this equation generates the following static system
\begin{equation}
\left\{ 
\begin{array}{l}
\text{div}(\nabla u_{0}+\nabla ^{T}u_{0})-\nabla q=u^*_0\text{ \ \ \ in \ }\Omega _{f} \\ 
\text{div}(u_0)=0\\
h_1=h^*_0 \text{ \ \ \ in \ }\Gamma _{s}\\ 
-[\nu \cdot (\nabla u_{0}+\nabla ^{T}u_{0})]|_{\Gamma
_{s}}+\Delta _{\Gamma _{s}}(h_{0})+[\nu \cdot \sigma (w_0)]|_{\Gamma
_{s}}+(q+c_0^*) \nu =h^*_1\text{ \ \ \ in \ }\Gamma_{s}\\
w_1=w^*_0\text{ \ \ \ in \ }\Omega _{s}\\
\text{div}\sigma (w_{0})-w_{0}=w^*_1\text{ \ \ \ in \ }\Omega _{s}\\
u_0|_{\Gamma_f}=0,~~u_0|_{\Gamma_s}=h_1=w_1|_{\Gamma_s}
\end{array}%
\right.   \label{stars5}
\end{equation}
where $p=q+c^*_0$ is the associated pressure as described in $D(\mathbf{A}).$ Firstly, the fluid component $u_0$ of \eqref{stars5} and the pressure term $q$ can be recovered via the Stokes Theory (See \cite[pg 22, Theorem 2.4]{temam}) and the pair $\{u_0, q\}\in [\bh^{1}(\Omega _{f})\cap Null(\text{div})]\times \hat{L}^{2}(\Omega _{f})$ solves the following static problem
\begin{equation}
\left\{ 
\begin{array}{l}
-\text{div}(\nabla u_{0}+\nabla ^{T}u_{0})+\nabla
q=-u_{0}^{\ast }\text{\ \ \ in \ }\Omega _{f} \\ 
\text{div}(u_{0})=0\text{\ \ \ \ in \ }\Omega _{f} \\ 
u_{0}|_{\Gamma_s}=h^*_0,%
\end{array}%
\right.   \label{res-2}
\end{equation}
with the estimate
\begin{equation}
||\nu \cdot (\nabla u_{0}+\nabla ^{T}u_{0})-q\nu||_{H^{-1/2}(\Gamma_s)}+||\nabla u_{0}+\nabla ^{T}u_{0}||_{\Omega_f}+||q||_{\Omega_f} \leq C||\Phi^*||_{\mathbf{H}}.  
 \label{res-3}
\end{equation}
Now, with the associated pressure $p=q+c^*_0,$ the constant component $c^*_0$ is to be determined. We turn our attention to the thick and thin elastic PDE component in \eqref{stars5}. Define the space%
\[
\textbf{S}=\left\{ (\varphi ,\psi )\in \bh^{1}(\Gamma _{s})\times \bh^{1}(\Omega
_{s}):\varphi =\psi |_{\Gamma _{s}}\right\}.
\]%
In order to generate a mixed variational formulation for the static ``thin"
and ``thick" solution variables in \eqref{stars5}, we respectively multiply (\ref{stars5})$_{4}$ and (\ref{stars5})$_{6}$ by functions $\varphi \in [\bh^{1}(\Gamma _{s})],$ and $\psi \in
[\bh^{1}(\Omega _{s})]$ in the space $\textbf{S}$, use Green's Theorem, and add the subsequent relations to get:%
\[
-\left\langle \nabla _{\Gamma _{s}}h_{0},\nabla _{\Gamma
_{s}}\varphi \right\rangle _{\Gamma _{s}}+\left\langle \nu \sigma(w_0)|_{\Gamma _{s}},\varphi \right\rangle _{\Gamma _{s}}+\left\langle c_0^*
\nu ,\varphi \right\rangle _{\Gamma _{s}} 
\]
\[
-\left\langle \nu \sigma(w_0)|_{\Gamma _{s}},\psi \right\rangle _{\Gamma _{s}}-\left\langle \sigma(w_0),\epsilon(\psi) \right\rangle _{\Omega _{s}}-\left\langle w_0,\psi \right\rangle _{\Omega _{s}}
\]
\begin{equation}
=\left\langle h_{1}^{\ast },\varphi \right\rangle _{\Gamma
_{s}}+\left\langle w_{1}^{\ast },\psi \right\rangle _{\Omega _{s}}+\left\langle \nu \cdot (\nabla u_{0}+\nabla ^{T}u_{0})|_{\Gamma
_{s}},\varphi \right\rangle _{\Gamma _{s}}-\left\langle q
\nu ,\varphi \right\rangle _{\Gamma _{s}}. 
\label{star3.5}
\end{equation}%
The last relation now gives us the following mixed variational
formulation in terms of the variables $h_{0}$ and $%
w_{0}:$ Namely,%

\begin{equation}
\left\{ 
\begin{array}{l}
\mathbf{a}([h_{0},w_{0}],[\varphi ,\psi ])+\mathbf{b}([\varphi ,\psi ],c^*_{0}) =\mathbf{F}([\varphi
,\psi ]),\text{ \ \ \ }\forall \text{ }[\varphi ,\psi ]\in \textbf{S} \\
\mathbf{b}([h_{0},w_{0}],r) =0,\text{ \ \ \ \ \ \ \ \ \ \ \ \ \ \ }\forall \text{ }r\in 
%TCIMACRO{\U{211d} }%
%BeginExpansion
\mathbb{R}.
%EndExpansion
\end{array}%
\right.   \label{star4}
\end{equation}
Here, the bilinear forms $\mathbf{a}(.,.):\textbf{S}\times \textbf{S}\rightarrow 
%TCIMACRO{\U{211d} }%
%BeginExpansion
\mathbb{R}
%EndExpansion
$ and $\mathbf{b}(.,.):\textbf{S}\times 
%TCIMACRO{\U{211d} }%
%BeginExpansion
\mathbb{R}
%EndExpansion
\rightarrow 
%TCIMACRO{\U{211d} }%
%BeginExpansion
\mathbb{R}
%EndExpansion
$ are respectively given as%
\begin{eqnarray*}
\mathbf{a}([\phi ,\xi ],[\widetilde{\phi },\widetilde{\xi }]) &=&\left\langle \nabla _{\Gamma _{s}}h_{0},\nabla _{\Gamma
_{s}}\varphi \right\rangle _{\Gamma _{s}}-\left\langle \nu \sigma(w_0)|_{\Gamma _{s}},\varphi \right\rangle _{\Gamma _{s}} \\
&&\hspace{-0.5cm}+\left\langle \nu \sigma(w_0)|_{\Gamma _{s}},\psi \right\rangle _{\Gamma _{s}}+\left\langle \sigma(w_0),\epsilon(\psi) \right\rangle _{\Omega _{s}}+\left\langle w_0,\psi \right\rangle _{\Omega _{s}},\end{eqnarray*}%
\[
\mathbf{b}([\widetilde{\phi },\widetilde{\xi }],r)=-r\int_{\Gamma_s} \nu \cdot \widetilde{\phi } d\Gamma_s,\]%
and the functional $\mathbf{F}(.)$ is defined as  
\begin{equation*}
\mathbf{F}([\widetilde{\phi },\widetilde{\xi }])=-\left\langle h_{1}^{\ast },\varphi \right\rangle _{\Gamma
_{s}}-\left\langle w_{1}^{\ast },\psi \right\rangle _{\Omega _{s}}-\left\langle \nu \cdot (\nabla u_{0}+\nabla ^{T}u_{0})|_{\Gamma
_{s}},\varphi \right\rangle _{\Gamma _{s}}+\left\langle q
\nu ,\varphi \right\rangle _{\Gamma _{s}}. 
\end{equation*}%
In order to solve this variational formulation we appeal to the Babuska-Brezzi Theorem (see Appendix, Theorem \ref{BB}). It is clear that the bilinear forms $\mathbf{a}(.,.)$ and $\mathbf{b}(.,.)$ are continuous, and moreover $\mathbf{a}(.,.)$ is $\textbf{S}$-elliptic. In order to conclude that the variational problem
(\ref{star4}) has a unique solution, we need to show that the bilinear form $%
\mathbf{b}(.,.)$ satisfies the ``inf-sup" condition given in Theorem \ref{BB}. For this, we consider the following
problem: \\

\noindent Given $r\in 
%TCIMACRO{\U{211d} }%
%BeginExpansion
\mathbb{R}
%EndExpansion
$, let $\eta\in \bh^{1}(\Gamma _{s})$  satisfy%
\[
\Delta _{\Gamma _{s}}\eta=sgn(r)\nu \text{ \ \ \ in \ }\Gamma _{s}
\]%
It is easily seen that $\left\Vert \nabla _{\Gamma _{s}}\eta\right\Vert _{\Gamma
_{s}}\leq C\left\Vert \nu \right\Vert _{\Gamma _{s}}.$ Now, taking into account that $\gamma
: \bh^1(\Omega_s)\rightarrow \bh^{1/2}(\Gamma_s)$ is a surjective map, and so it has a continuous right inverse $\gamma^+(\eta),$ we have%
\begin{eqnarray*}
\sup_{\lbrack \theta ,\varsigma ]\in \textbf{S}}\frac{b([\theta ,\varsigma ],r)}{%
\left\Vert [\theta ,\varsigma ]\right\Vert _{\textbf{S}}} &\geq &\frac{b([\eta,\gamma^+(\eta)],r)}{%
\left\Vert [\eta,\gamma^+(\eta)]\right\Vert _{\textbf{S}}} \\
&=&\frac{-r\int\limits_{\Gamma _{s}}\nu \cdot \eta d\Gamma _{s}}{\left\Vert [\eta,\gamma^+(\eta)]\right\Vert _{\textbf{S}}} \\
&=&-rsgn(r)\frac{\int\limits_{\Gamma _{s}}\Delta _{\Gamma _{s}}\eta\cdot
\eta d\Gamma _{s}}{\left\Vert [\eta,\gamma^+(\eta)]\right\Vert _{\textbf{S}}} \\
&=&\left\vert r\right\vert \frac{\int\limits_{\Gamma _{s}}|\nabla _{\Gamma
_{s}}\eta|^{2}d\Gamma _{s}}{\left\Vert [\eta,\gamma^+(\eta)]\right\Vert _{\textbf{S}}} \\
&\geq&C\left\vert r\right\vert \frac{\int\limits_{\Gamma _{s}}|\nabla _{\Gamma
_{s}}\eta|^{2}d\Gamma _{s}}{\left\Vert \eta \right\Vert _{\bh^1(\Gamma_s)}}\\
&=&C\left\vert r\right\vert \left\Vert \eta\right\Vert _{\bh^{1}(\Gamma _{s})},
\end{eqnarray*}%
which yields that the inf-sup condition holds with the constant $\beta =C\left\Vert \eta\right\Vert
_{\bh^{1}(\Gamma _{s})}.$ Consequently, the existence and uniqueness of the solution $%
[h_{0},w_{0}]\in \textbf{S}$ and $c^*_{0}\in 
%TCIMACRO{\U{211d} }%
%BeginExpansion
\mathbb{R}
%EndExpansion
$  to the mixed variational problem (\ref{star4}) follows from Theorem \ref{BB}, and satisfy
\begin{equation}
||[h_{0},w_{0}]||_{\textbf{S}}+|c^*_0|\leq C||\Phi^*||_{\mathbf{H}}. \label{star4.5}
\end{equation}
Subsequently, if we take $[\varphi ,\psi ]=[0,\psi]$ in \eqref{star4} where $\psi \in [D(\Omega_s)]^3,$ we infer that the obtained $w_0$ solves:
\begin{equation}
-\text{div}\sigma (w_{0})+w_{0}=-w^*_1\text{ \ \ \ in \ }\Omega _{s}. \label{star5}
\end{equation}
In turn, via an energy method, we have the estimate
\begin{equation}
||\nu \cdot \sigma(w_0))||_{H^{-1/2}(\Gamma_s)} \leq C ||\Phi^*_0||_{\mathbf{H}}. \label{star6}
\end{equation}
With this estimate in hand, $\{[h_{0},w_{0}],c^*_{0}\}$ solves
\[
\left\langle \nabla _{\Gamma _{s}}h_{0},\nabla _{\Gamma
_{s}}\varphi \right\rangle _{\Gamma _{s}}+\left\langle \sigma(w_0),\epsilon(\psi) \right\rangle _{\Omega _{s}}+\left\langle w_0,\psi \right\rangle _{\Omega _{s}}-\left\langle c^*_0
\nu ,\varphi \right\rangle _{\Gamma _{s}} \]
\begin{equation*}
=-\left\langle h_{1}^{\ast },\varphi \right\rangle _{\Gamma
_{s}}-\left\langle w_{1}^{\ast },\psi \right\rangle _{\Omega _{s}}-\left\langle \nu \cdot (\nabla u_{0}+\nabla ^{T}u_{0})|_{\Gamma
_{s}},\varphi \right\rangle _{\Gamma _{s}}+\left\langle q
\nu ,\varphi \right\rangle _{\Gamma _{s}}.
\end{equation*}%
An integration by parts and consideration of \eqref{star5} then yields
\[
-\Delta _{\Gamma _{s}}(h_{0})+[\nu \cdot (\nabla u_{0}+\nabla ^{T}u_{0})]|_{\Gamma
_{s}}-[\nu \cdot \sigma (w_0)]|_{\Gamma
_{s}}-p \nu =-h^*_1\text{ \ \ \ in \ }\Gamma_{s}.
\]
If we read off the last equation then we get
\begin{equation*}
||-\Delta _{\Gamma _{s}}(h_{0})+[\nu \cdot (\nabla u_{0}+\nabla ^{T}u_{0})]|_{\Gamma
_{s}}-[\nu \cdot \sigma (w_0)]|_{\Gamma
_{s}}-p \nu ||_{\Gamma_{s}}
\leq C||\Phi^*||_{\mathbf{H}}.
\end{equation*}
Finally, the pressure term $p$ can be reconstructed via the maps $(\mathcal{P}_{i})s,$ as defined in \eqref{11.5}, which then yields that $\Phi =\left[
u_{0},h_{0},h_{1},w_{0},w_{1}\right] \in D(\mathbf{A})\cap N^{\bot}$ indeed solves \eqref{star}. Hence $0 \in \rho(\mathbf{A}|_{N^{\bot}}).$
\end{proof}
\\

\noindent As a result, combining Theorem \ref{eigen}, Lemma \ref{point}, Lemma \ref{res}, Lemma \ref{cont} and Lemma \ref{resolv} finishes the proof of our main result Theorem \ref{main}.

\section{Appendix}

For the reader’s convenience, we provide the statements of the following theorems which are critically used in this manuscript.

\begin{theorem} \cite{A-B}
\label{sta} 
Let ${T(t)}_{t\geq 0}$ be a bounded $C_{0}$-semigroup on a reflexive Banach space X, with generator $\mathbf{A}$. Assume that $\sigma_p(\mathbf{A})\cap i\mathbb{R}=\emptyset$, where $\sigma_p(\mathbf{A})$ is the point spectrum of $\mathbf{A}$. If $\sigma(\mathbf{A})\cap i\mathbb{R}$ is countable then ${T(t)}_{t\geq 0}$ is strongly stable.
\end{theorem}
In order to establish our strong stability result, one of the key tools that we use to show that zero is in the resolvent set of the operator $\{\mathbf{A}|_{N^{\bot}}\}$ is the Babuska-Brezzi Theorem whose statement here is recalled for the reader's convenience:
\begin{theorem}
\cite{kesevan}\label{BB} (Babuska-Brezzi) Let $\Sigma ,$ $V$ be Hilbert spaces and $a:\Sigma \times
\Sigma \rightarrow 
%TCIMACRO{\U{211d} }%
%BeginExpansion
\mathbb{R}
%EndExpansion
,$ $b:\Sigma \times V\rightarrow 
%TCIMACRO{\U{211d} }%
%BeginExpansion
\mathbb{R}
%EndExpansion
$ bilinear forms which are continuous. Let%
\[
Z=\left\{ \sigma \in \Sigma :b(\sigma ,v)=0,\text{ \ for every }v\in
V\right\} .
\]%
Assume that $a(\cdot ,\cdot )$ is $Z$-elliptic, i.e., there exists a
constant $\alpha >0$ such that 
\[
a(\sigma ,\sigma )\geq \alpha \left\Vert \sigma \right\Vert _{\Sigma }^{2},%
\text{ \ \ for every }\alpha \in Z.\text{ }
\]%
Assume further that there exists a constant $\beta >0$ such that%
\[
\sup_{\tau \in \Sigma }\frac{b(\tau ,v)}{\left\Vert \tau \right\Vert
_{\Sigma }}\geq \beta \left\Vert v\right\Vert _{V},\text{ \ \ for every }%
v\in V.
\]%
Then if $\kappa \in \Sigma $ and $l\in V,$ there exists a unique pair $%
(\sigma ,u)\in \Sigma \times V$ such that%
\begin{equation*}
\left\{ 
\begin{array}{l}
a(\sigma ,\tau )+b(\tau ,u)=(\kappa ,\tau )\text{ \ \ \ }\forall \text{ }%
\tau \in \Sigma \\

b(\sigma ,v)=(l,v)\text{ \ \ \ \ \ \ \ \ \ \ \ \ \ \ \ \ }\forall \text{ }v\in V.
\end{array}
\right. \label{12.5}
\end{equation*}
\end{theorem}

%\section{Data Availability}

%The author declares that she does not analyze or generate any datasets, because her work proceeds within a theoretical and mathematical approach. 

%\section{Statements and Declarations}

%The author declares that there is no financial or non-financial interests that are directly or indirectly related to this work.

\section{Acknowledgement}

\noindent The author Pelin G. Geredeli would like to thank the National Science Foundation, and acknowledge her partial funding from NSF Grant DMS-2348312.


\begin{thebibliography}{99}

\bibitem{A-B} W. Arendt and C. J. K. Batty, Tauberian theorems and stability of one-parameter semigroups, \textit{Trans. Amer. Math. Soc.} 306, (1988) pp. 837--852.
%\bibitem{AD} G. Avalos and M. Dvorak A new maximality argument for a coupled fluid-structure interaction with implications for a divergence-free finite element method, \textit{Applicationes Mathematicae,} Vol. 35 (3), (2008), 259-280 

\bibitem{Av} G. Avalos; The strong stability and instability of a fluid-structure semigroup, \textit{Appl, Math. Optim.} (55) (2007), pp. 163-184.

\bibitem{GR} G. Avalos and R. Triggiani; Boundary feedback stabilization of a coupled parabolic-hyperbolic Stokes-Lam-e PDE system, \textit{Journal of Evolution Equations} (9) (2009), pp. 341-370.

\bibitem{av-trig} G. Avalos and R. Triggiani, The Coupled PDE System Arising in Fluid-Structure Interaction, Part I: Explicit Semigroup Generator and its Spectral Properties, \textit{Contemporary Mathematics}, Volume 440, (2007), pp. 15-54.

%\bibitem{AvalosLasieckaTriggiani16}
%George Avalos, Irena Lasiecka, and Roberto Triggiani.
%\newblock Heat-wave interaction in 2--3 dimensions: optimal rational decay
  %rate.
%\newblock {\em J. Math. Anal. Appl.}, 437(2):782--815, 2016.

\bibitem{AvalosTriggiani09} G. Avalos and R. Triggiani.
\newblock Semigroup well-posedness in the energy space of a
  parabolic-hyperbolic coupled {S}tokes-{L}am\'{e} {PDE} system of
  fluid-structure interaction.
\newblock {\em Discrete Contin. Dyn. Syst. Ser. S}, 2(3): (2009), pp. 417--447.

%\bibitem{AvalosTriggiani2}
%George Avalos and Roberto Triggiani.
%\newblock Rational decay rates for a {PDE} heat-structure interaction: a
 % frequency domain approach.
%\newblock {\em Evol. Equ. Control Theory}, 2(2):233--253, 2013.

\bibitem{AGM} G. Avalos, P. G. Geredeli, B. Muha; \textquotedblleft
Wellposedness, Spectral Analysis and Asymptotic Stability of a Multilayered
Heat-Wave-Wave System \textquotedblright, \emph{Journal of Differential Equations} 269 (2020), pp. 7129-7156.

\bibitem{AA} G. Avalos, P. G. Geredeli; \textquotedblleft
A Resolvent Criterion Approach to Strong Decay of a Multilayered Lamé-Heat System\textquotedblright, \emph{Applicable Analysis,}  https://doi.org/10.1080/00036811.2021.1954165  (2021).

\bibitem{RD} G. Avalos, P. G. Geredeli and B. Muha; Rational Decay of A Multilayered Structure-Fluid PDE System,\textit{ Journal of Mathematical Analysis and Applications 514 (2),\\ https://doi.org/10.1016/j.jmaa.2022.126284} (2022)

\bibitem{Barbu} V. Barbu, Z. Grujić, I. Lasiecka and A. Tuffaha, "Weak and strong solutions in nonlinear fluid-structure interactions", \textit{Contemporary Mathematics} 440: Fluids and Waves, American Mathematical Society, Providence Rhode Island (2007), pp. 55-82.


%\bibitem{Batty19}
%Charles Batty, Lassi Paunonen, and David Seifert.
%\newblock Optimal energy decay for the wave-heat system on a rectangular
%  domain.
%\newblock {\em SIAM J. Math. Anal.}, 51(2):808--819, 2019.

\bibitem{ap-1} Y. Bazilevs, M-C. Hsu, Y. Zhang, W. Wang, T. Kvamsdal, S. Hentschel, J. Isaksen; \textquotedblleft Computational fluid-structure interaction: methods and application to cerebral aneurysms.\textquotedblright , \emph{Biomech Model Machanobiol} 9 (2010), pp. 481-498.

\bibitem{BL} A. Benadallah, I. Lasiecka; Exponential Decay Rates for a Full von Karman System of Dynamic Thermoelasticity, \textit{Journal of Differential Equations} 160 (2000), pp. 51-93. 

\bibitem{bench} C.D. Benchimol; A note on weak stability of contraction semigroups, \textit{SIAM J. Control and Optimization,} Vol. 16 (3), May 1978, pp. 373-379


\bibitem{FSIforBIO}
T. Bodnar, G.~P. Galdi, and S. Necasova, editors.
\newblock {\em Fluid-Structure Interaction and Biomedical Applications}.
\newblock Birkh{\"a}user/Springer, Basel, 2014.

%\bibitem{BF} F. Brezzi and M. Fortin; (1991) \textit{Mixed and Hybrid Finite Element Methods,} Springer-Verlag Newyork Inc

\bibitem{multi-layered}
M.~Bukac, S.~Canic, and B.~Muha.
\newblock A partitioned scheme for fluid-composite structure interaction
  problems.
\newblock {\em Journal of Computational Physics}, 281(0): (2015) pp. 493-517.

\bibitem{buk} M. Bukac, S. Canic, B. Muha; \textquotedblleft A nonlinear fluid-structure interaction problem in compliant arteries treated with vascular stents,\textquotedblright \emph{Applied Mathematics and Optimization,} 73 (3), (2016), pp. 433-473.

%\bibitem{bodnar2017particles}
%Tom{\'a}{\v{s}} Bodn{\'a}r, Giovanni~P Galdi, and {\v{S}}{\'a}rka
 % Ne{\v{c}}asov{\'a}.
%\newblock {\em Particles in Flows}.
%\newblock Springer, 2017.


%\bibitem{SunnyStentsSIAM}
%S~{\v C}ani{\' c}.
%\newblock New mathematics for next-generation stent design.
%\newblock {\em SIAM News, April 2019}.

\bibitem{Chambolle} A. Chambolle, B. Desjardins, M.J. Esteban and C. Grandmont, Existence of weak solutions for the unsteady interaction of a viscous fluid with an elastic plate, \textit{J. Math. Fluid Mech }7, no. 3, (2005), pp. 368-404.
 

%\bibitem{chen} G. Chen, A note on the boundary stabilization of the wave equation, \textit{SIAM J. Control Optim.} 19: 106--113, 1981.

\bibitem{ciarlet}
Philippe~G. Ciarlet.
\newblock {\em Mathematical elasticity. {V}ol. {I}}: Three-Dimensional Elasticity \newblock North-Holland Publishing Co., Amsterdam, 1988.


%\bibitem{ciarlet} P. G. Ciarlet, \textit{Numerical Analysis of the Finite
%Element Method}, \textit{S\'{e}minaire de Math\'{e}matique Sup\'{e}rieures},
%Les Presses de L'Universit\'{e} de Montr\'{e}al (1976).

\bibitem{Courtand} D. Coutand and S. Shkoller, ``Motion of an elastic solid inside an incompressible viscous fluid’’, \textit{Arch. Ration. Mech. Anal.} 176 (2005), pp. 25-102.
 
 \bibitem{Ger} P. G. Geredeli;  An inf-sup Approach to $C_0$-Semigroup Generation for An Interactive Composite Structure-Stokes PDE Dynamics,\textit{ Journal of Evolution Equations,} \textit{In press} https://arxiv.org/pdf/2401.13962.pdf, (2024)
 
\bibitem{Du} Q. Du, M. D. Gunzburger, L. S. Hou and J. Lee, "Analysis of a linear fluid-structure interaction problem", \textit{Discrete and Continuous Dynamical Systems}, Volume 9, No. 3 (May 2003), pp. 633-650.

%\bibitem{GazzolaReview}
%Filippo~Gazzola Denis~Bonheure and Gianmarco Sperone.
%\newblock Eight(y) mathematical questions on fluids and structures.
%\newblock {\em preprint, http://www1.mate.polimi.it/~gazzola/turbulence.pdf}.

%\bibitem{Dowell15}
%Earl~H. Dowell.
%\newblock {\em A modern course in aeroelasticity}, volume 217 of {\em Solid
  %Mechanics and its Applications}.
%\newblock Springer, Cham, enlarged edition, 2015.

%\bibitem{Duyckaerts}
%Thomas Duyckaerts.
%\newblock Optimal decay rates of the energy of a hyperbolic-parabolic system
  %coupled by an interface.
%\newblock {\em Asymptot. Anal.}, 51(1):17--45, 2007.

\bibitem{F} B. Friedman, Principles and Techniques of Applied Mathematics, Dover Publications, Inc., New York, 1990.

%\bibitem{GV} P. Grisvard, Elliptic Problems in Nonsmooth Domains, Pitman Publishing Inc., Boston  1985.

%\bibitem{Grisvard2} P. Grisvard, \textit{Singularities in Boundary Value
%Problems, Research Notes in Applied Mathematics 22}, Springer-Verlag, New
%York (1992).

%\bibitem{huang} Huang, F.L. Characteristic conditions for exponential
%stability of linear dynamical systems in Hilbert spaces, \emph{Ann. Differ.
%Equ.}, 1(1), pp. 43-53 (1985).

%\bibitem{Fathallah}
%Ines~Kamoun Fathallah.
%\newblock Logarithmic decay of the energy for an hyperbolic-parabolic coupled
  %system.
%\newblock {\em ESAIM Control Optim. Calc. Var.}, 17(3):801--835, 2011.

%\bibitem{HansenZuazua}
%Scott Hansen and Enrique Zuazua.
%\newblock Exact controllability and stabilization of a vibrating string with an
  %interior point mass.
%\newblock {\em SIAM J. Control Optim.}, 33(5):1357--1391, 1995.

\bibitem{horn} M. A. Horn; Implications of Sharp Trace Regularity Results on Boundary Stabilization of the System of Linear Elasticity, \textit{Journal of Mathematical Analysis and Applications} 223, (1998), pp. 126-150.

\bibitem{hsu} M-C. Hsu, Y. Bazilevs;  \textquotedblleft Blood vessel tissue prestress modeling for vascular fluid–structure interaction simulation\textquotedblright  \emph{Finite Elements in Analysis and Design,} 47, (2011), pp. 593-599.

\bibitem{kesevan} S. Kesevan, \emph{Topics in Functional Analysis and
Applications}, New Age International Limited Publishers, New Delhi (1989).

\bibitem{KochZauzua} H. Koch and E. Zuazua.
\newblock A hybrid system of {PDE}'s arising in multi-structure interaction:
  coupling of wave equations in {$n$} and {$n-1$} space dimensions.
\newblock In {\em Recent trends in partial differential equations}, volume 409
  of {\em Contemp. Math.}, pages 55--77. Amer. Math. Soc., Providence, RI,
  2006.

\bibitem{lions1969quelques}
Jacques-Louis and Lions.
\newblock {\em Quelques m{\'e}thodes de r{\'e}solution des problemes aux
  limites non lin{\'e}aires}, volume~31.
\newblock Dunod Paris, 1969.

%\bibitem{MTEFSI}
%Barbara Kaltenbacher, Igor Kukavica, Irena Lasiecka, Roberto Triggiani, Amjad
  %Tuffaha, and Justin~T. Webster.
%\newblock {\em Mathematical theory of evolutionary fluid-flow structure
  %interactions}, volume~48 of {\em Oberwolfach Seminars}.
%\newblock Birkh\"{a}user/Springer, Cham, 2018.
%\newblock Lecture notes from Oberwolfach seminars, November 20--26, 2016.



%\bibitem{LescarretZuazua15}
%Vincent Lescarret and Enrique Zuazua.
%\newblock Numerical approximation schemes for multi-dimensional wave equations
  %in asymmetric spaces.
%\newblock {\em Math. Comp.}, 84(291):119--152, 2015.

%\bibitem{L-P} Y.I. Lyubich and V.Q. Phong, \textquotedblleft Asymptotic
%stability of linear differential equations in Banach
%Spaces\textquotedblright , Studia Matematica, LXXXVII (1988), pp. 37-42.

%\bibitem{mclean} W. McLean, \textit{Strongly Elliptic Systems and Boundary
%Integral Equations}, Cambridge University Press, NewYork (2000).

\bibitem{BorisSimplifiedFSI}
B. Muha.
\newblock A note on optimal regularity and regularizing effects of point mass
  coupling for a heat-wave system.
\newblock {\em Journal of Mathematical Analysis and Applications}, 425 (2), 2015, pp. 1134-1147.

\bibitem{SunBorMulti}
B. Muha and S. {\v{C}}ani{\'c}.
\newblock Existence of a solution to a fluid--multi-layered-structure
  interaction problem.
\newblock {\em J. Differential Equations}, 256(2), (2014), pp. 658--706.

\bibitem{Necas}  J. Nečas, Direct Methods in the Theory of Elliptic Equations, Springer, New York, 2012 (translated by Gerard Tronel and
Alois Kufner).

\bibitem{RauchZhangZuazua}
J. Rauch, X.~Zhang, and E. Zuazua.
\newblock Polynomial decay for a hyperbolic-parabolic coupled system.
\newblock {\em J. Math. Pures Appl. (9)}, 84(4), (2005), pp. 407-470.

%\bibitem{RichterBook}
%Thomas Richter.
%\newblock {\em Fluid-structure interactions}, volume 118 of {\em Lecture Notes
  %in Computational Science and Engineering}.
%\newblock Springer, Cham, 2017.
%\newblock Models, analysis and finite elements.

\bibitem{temam} R. Temam; \emph{Navier Stokes Equations, Theory and
Numerical Analysis}, AMS Chelsea Publishing, Providence, R.I (2001).

%\bibitem{trigg} R. Triggiani, Wave equation on a bounded domain with boundary dissipation: an operator approach, \textit{J. Math. Anal. Appl.} 137: 438-461, 1989.

%\bibitem{ZhangZuazuaARMA07}
%Xu~Zhang and Enrique Zuazua.
%\newblock Long-time behavior of a coupled heat-wave system arising in
  %fluid-structure interaction.
%\newblock {\em Arch. Ration. Mech. Anal.}, 184(1):49--120, 2007.

\bibitem{tomilov} Y. Tomilov, A resolvent approach to stability of operator semigroups, \emph{J. Operator Theory,} 46,(2001), pp. 63--98.

\end{thebibliography}
\end{document}